\documentclass[11pt]{amsart}

\usepackage{amssymb}
\usepackage{amsmath}
\usepackage[all,cmtip]{xy}
\usepackage{amsfonts}
\usepackage{mathrsfs}
\usepackage{latexsym}
\usepackage{graphicx}
\usepackage{amscd,amssymb,amsmath,amsbsy,amsthm}
\usepackage[colorlinks,plainpages,backref,urlcolor=blue]{hyperref}
\allowdisplaybreaks

\usepackage{tikz}
\usetikzlibrary{arrows,calc}
\tikzset{
>=stealth',
help lines/.style={dashed, thick},
axis/.style={<->},
important line/.style={thick},
connection/.style={thick, dotted},
}

\newcommand{\nc}{\newcommand}
\nc{\rnc}{\renewcommand}
\nc{\bb}[1]{{\mathbb #1}}
\nc{\bbA}{\bb{A}}\nc{\bbB}{\bb{B}}\nc{\bbC}{\bb{C}}\nc{\bbD}{\bb{D}}
\nc{\bbE}{\bb{E}}\nc{\bbF}{\bb{F}}\nc{\bbG}{\bb{G}}\nc{\bbH}{\bb{H}}
\nc{\bbI}{\bb{I}}\nc{\bbJ}{\bb{J}}\nc{\bbK}{\bb{K}}\nc{\bbL}{\bb{L}}
\nc{\bbM}{\bb{M}}\nc{\bbN}{\bb{N}}\nc{\bbO}{\bb{O}}\nc{\bbP}{\bb{P}}
\nc{\bbQ}{\bb{Q}}\nc{\bbR}{\bb{R}}\nc{\bbS}{\bb{S}}\nc{\bbT}{\bb{T}}
\nc{\bbU}{\bb{U}}\nc{\bbV}{\bb{V}}\nc{\bbW}{\bb{W}}\nc{\bbX}{\bb{X}}
\nc{\bbY}{\bb{Y}}\nc{\bbZ}{\bb{Z}}
\nc{\mbf}[1]{{\mathbf #1}}
\nc{\bfA}{\mbf{A}}\nc{\bfB}{\mbf{B}}\nc{\bfC}{\mbf{C}}\nc{\bfD}{\mbf{D}}
\nc{\bfE}{\mbf{E}}\nc{\bfF}{\mbf{F}}\nc{\bfG}{\mbf{G}}\nc{\bfH}{\mbf{H}}
\nc{\bfI}{\mbf{I}}\nc{\bfJ}{\mbf{J}}\nc{\bfK}{\mbf{K}}\nc{\bfL}{\mbf{L}}
\nc{\bfM}{\mbf{M}}\nc{\bfN}{\mbf{N}}\nc{\bfO}{\mbf{O}}\nc{\bfP}{\mbf{P}}
\nc{\bfQ}{\mbf{Q}}\nc{\bfR}{\mbf{R}}\nc{\bfS}{\mbf{S}}\nc{\bfT}{\mbf{T}}
\nc{\bfU}{\mbf{U}}\nc{\bfV}{\mbf{V}}\nc{\bfW}{\mbf{W}}\nc{\bfX}{\mbf{X}}
\nc{\bfY}{\mbf{Y}}\nc{\bfZ}{\mbf{Z}}
\nc{\bfa}{\mbf{a}}\nc{\bfb}{\mbf{b}}\nc{\bfc}{\mbf{c}}\nc{\bfd}{\mbf{d}}
\nc{\bfe}{\mbf{e}}\nc{\bff}{\mbf{f}}\nc{\bfg}{\mbf{g}}\nc{\bfh}{\mbf{h}}
\nc{\bfi}{\mbf{i}}\nc{\bfj}{\mbf{j}}\nc{\bfk}{\mbf{k}}\nc{\bfl}{\mbf{l}}
\nc{\bfm}{\mbf{m}}\nc{\bfn}{\mbf{n}}\nc{\bfo}{\mbf{o}}\nc{\bfp}{\mbf{p}}
\nc{\bfq}{\mbf{q}}\nc{\bfr}{\mbf{r}}\nc{\bfs}{\mbf{s}}\nc{\bft}{\mbf{t}}
\nc{\bfu}{\mbf{u}}\nc{\bfv}{\mbf{v}}\nc{\bfw}{\mbf{w}}\nc{\bfx}{\mbf{x}}
\nc{\bfy}{\mbf{y}}\nc{\bfz}{\mbf{z}}

\nc{\mcal}[1]{{\mathcal #1}}
\nc{\calA}{\mcal{A}}\nc{\calB}{\mcal{B}}\nc{\calC}{\mcal{C}}\nc{\calD}{\mcal{D}}
\nc{\calE}{\mcal{E}} \nc{\calF}{\mcal{F}}\nc{\calG}{\mcal{G}}\nc{\calH}{\mcal{H}}
\nc{\calI}{\mcal{I}}\nc{\calJ}{\mcal{J}}\nc{\calK}{\mcal{K}}\nc{\calL}{\mcal{L}}
\nc{\calM}{\mcal{M}}\nc{\calN}{\mcal{N}}\nc{\calO}{\mcal{O}}\nc{\calP}{\mcal{P}}
\nc{\calQ}{\mcal{Q}}\nc{\calR}{\mcal{R}}\nc{\calS}{\mcal{S}}\nc{\calT}{\mcal{T}}
\nc{\calU}{\mcal{U}}\nc{\calV}{\mcal{V}}\nc{\calW}{\mcal{W}}\nc{\calX}{\mcal{X}}
\nc{\calY}{\mcal{Y}}\nc{\calZ}{\mcal{Z}}
\nc{\fA}{\frak{A}}\nc{\fB}{\frak{B}}\nc{\fC}{\frak{C}} \nc{\fD}{\frak{D}}
\nc{\fE}{\frak{E}}\nc{\fF}{\frak{F}}\nc{\fG}{\frak{G}}\nc{\fH}{\frak{H}}
\nc{\fI}{\frak{I}}\nc{\fJ}{\frak{J}}\nc{\fK}{\frak{K}}\nc{\fL}{\frak{L}}
\nc{\fM}{\frak{M}}\nc{\fN}{\frak{N}}\nc{\fO}{\frak{O}}\nc{\fP}{\frak{P}}
\nc{\fQ}{\frak{Q}}\nc{\fR}{\frak{R}}\nc{\fS}{\frak{S}}\nc{\fT}{\frak{T}}
\nc{\fU}{\frak{U}}\nc{\fV}{\frak{V}}\nc{\fW}{\frak{W}}\nc{\fX}{\frak{X}}
\nc{\fY}{\frak{Y}}\nc{\fZ}{\frak{Z}}
\nc{\fa}{\frak{a}}\nc{\fb}{\frak{b}}\nc{\fc}{\frak{c}} \nc{\fd}{\frak{d}}
\nc{\fe}{\frak{e}}\nc{\fFf}{\frak{f}}\nc{\fg}{\frak{g}}\nc{\fh}{\frak{h}}
\nc{\fri}{\frak{i}}\nc{\fj}{\frak{j}}\nc{\fk}{\frak{k}}\nc{\fl}{\frak{l}}
\nc{\fm}{\frak{m}}\nc{\fn}{\frak{n}}\nc{\fo}{\frak{o}}\nc{\fp}{\frak{p}}
\nc{\fq}{\frak{q}}\nc{\fr}{\frak{r}}\nc{\fs}{\frak{s}}\nc{\ft}{\frak{t}}
\nc{\fu}{\frak{u}}\nc{\fv}{\frak{v}}\nc{\fw}{\frak{w}}\nc{\fx}{\frak{x}}
\nc{\fy}{\frak{y}}\nc{\fz}{\frak{z}}

\newtheorem{theorem}{Theorem}[section]
\newtheorem{lemma}[theorem]{Lemma}
\newtheorem{corollary}[theorem]{Corollary}
\newtheorem{prop}[theorem]{Proposition}

\theoremstyle{definition}
\newtheorem{definition}[theorem]{Definition}
\newtheorem{example}[theorem]{Example}
\newtheorem{remark}[theorem]{Remark}

 \DeclareMathOperator{\gr}{gr}
\DeclareMathOperator{\im}{im} 
\DeclareMathOperator{\codim}{codim} \DeclareMathOperator{\id}{id}

\DeclareMathOperator{\Ext}{{Ext}}
\DeclareMathOperator{\sHom}{{\mathscr{H}om}}
\DeclareMathOperator{\sExt}{{\mathscr{E}xt}}

\DeclareMathOperator{\proj}{pr}

\DeclareMathOperator{\Spec}{{Spec}} 
\DeclareMathOperator{\Aut}{Aut}
 
\DeclareMathOperator{\sEnd}{{\mathscr{E}nd}}

\DeclareMathOperator{\Coh}{Coh}
\DeclareMathOperator{\Mod}{Mod\hbox{-}}

\DeclareMathOperator{\Pic}{Pic}

\DeclareMathOperator{\ad}{ad}

\newcommand{\catA}{\mathfrak{A}}

\DeclareMathOperator{\Res}{Res}

\newcommand{\inj}{\hookrightarrow}

\newcommand{\al}{\alpha}
\newcommand{\la}{\lambda}

\newcommand{\de}{\delta}

\newcommand{\reg}{\mathrm{reg}}

\newcommand{\dyn}{{\operatorname{d}}}

\newcommand{\frakp}{\mathfrak{p}}
\newcommand{\lag}{\langle}
\newcommand{\rag}{\rangle}

\newcommand{\tensor}{\star}
\newcommand{\dynst}{{\dyn*}}

\newcommand{\whFM}{\widehat{FM}}

\newcount\cols
{\catcode`,=\active\catcode`|=\active
 \gdef\Young(#1){\hbox{$\vcenter
 {\mathcode`,="8000\mathcode`|="8000
  \def,{\global\advance\cols by 1 &}%
  \def|{\cr
        \multispan{\the\cols}\hrulefill\cr
        &\global\cols=2 }%
  \offinterlineskip\everycr{}\tabskip=0pt
  \dimen0=\ht\strutbox \advance\dimen0 by \dp\strutbox
  \halign
   {\vrule height \ht\strutbox depth \dp\strutbox##
    &&\hbox to \dimen0{\hss$##$\hss}\vrule\cr
    \noalign{\hrule}&\global\cols=2 #1\crcr
    \multispan{\the\cols}\hrulefill\cr%
   }
 }$}}
}

\title[A Langlands duality of elliptic Hecke algebras]
{A Langlands duality of elliptic Hecke algebras}
\date{\today}
\author{Gufang Zhao}
\address{School of Mathematics and Statistics, University of Melbourne, Parkville VIC 3010, Australia}
\email{gufangz@unimelb.edu.au}
\author{Changlong Zhong}
\address{State University of New York at Albany, 1400 Washington Ave, CK  343 , Albany, NY 12222}
\email{czhong@albany.edu}

\begin{document}
\begin{abstract}
Associated to any root datum, there is an elliptic affine Hecke algebra defined by Ginzburg, Kapranov, and Vasserot. In this study, we establish a Fourier-Mukai functor from the representation category of the elliptic affine Hecke algebra to the corresponding category associated with the Langlands dual root datum. To achieve this connection, we employ the  elliptic Hecke algebra   with dynamical parameters as an intermediary.
\end{abstract}

\maketitle

\section{Introduction}
The notion of an elliptic affine Hecke algebra, associated to a root datum, was introduced by Ginzburg, Kapranov, and Vasserot \cite{GKV97} as a coherent sheaf on an abelian variety. In this paper, we refer to this algebra as the GKV elliptic Hecke algebra or simply the GKV algebra. Notably, a representation of this algebra has an underlying coherent sheaf on an abelian variety. This observation naturally leads us to the question of whether the Fourier-Mukai transform of the abelian variety can be extended to the representation categories of elliptic affine Hecke algebras. In this work, we provide an affirmative answer to this question.

One crucial component of our study involves the introduction of a  version of elliptic affine Hecke algebras with dynamical parameters. 
This algebra is a sheaf defined on the product of the aforementioned abelian variety and its dual. An essential feature of this algebra is that the correct notion of polynomial representation is given by (a translate of) the Poincar\'e line bundle.
Various components of this algebra have already appeared in literature. Notably, 
Demazure-Lusztig operators with dynamical parameters were originally discovered by Rim\'anyi and Weber in the context of elliptic Schubert calculus \cite{RW20}. A remarkable consequence of their findings is that these operators satisfy the braid relation. Building upon these ideas, in  \cite{ZZ15}, it was observed that these operators can be embedded within an elliptic Hecke algebra with dynamical parameters, a variant of the GKV algebra with the aforementioned polynomial representation, partially inspired by Felder's dynamical elliptic $R$-matrices as well as \cite{RW20}. 
However, for the purpose of representation theory, it is necessary to find a suitable integral structure. 
In \S~\ref{sec:res}, we provide a GKV-style definition of the elliptic Hecke algebra  with dynamical parameters (see Definition~\ref{def:res}), as well as a definition of the representation category associated to this algebra.

Utilizing the Poincar\'e line bundle, we construct a functor from the representation category of the GKV algebra to that of the dynamical elliptic Hecke algebra (see Theorem~\ref{thm:pullback}). Furthermore, we elucidate the Weyl group relations of the dynamical Demazure-Lusztig operators in terms of the Langlands dual root datum (see Theorem~\ref{thm:Weyl}). By combining these results, we demonstrate that the Fourier-Mukai transform on abelian varieties can be extended to the representation categories of GKV algebras (see Corollary ~\ref{cor:fm}).
Various consequences deduced using this functor are discussed in \S~\ref{sec:conclutions}.

\subsection*{Acknowledgements}
The question of Fourier-Mukai transform in this setting arose from the discussions between Valerio Toledano Laredo and G.Z. around 2015. We would like to give a special thanks to him for his insights in this project, and for discussions on related topics. The approach in the present paper by introducing dynamical parameters was conceived when both authors were collaborating with Cristian Lenart on \cite{LZZ23}. Much of the work in this project, discovering the dynamical Hecke algebra and its polynomial representation, is invisible in the short proof presented in this final product. The formulas of Rim\'anyi and Weber \cite{RW20} was pivotal to this invisible process. We were often asked about an explanation of the fact Rim\'anyi-Weber's DL-operators satisfy the Weyl group relations. 
The explanation in the present paper in terms of the Hecke algebra of the Langlands dual came as a pleasant surprise to us. 

G.~Z.~is partially supported by the Australian Research Council via DE190101222 and DP210103081. Part of this research was conducted when C. Z. was visiting the University of Melbourne,  the Sydney Mathematical Research Institute, University of Ottawa, and the Max-Planck Institute of Mathematics in Bonn. He would like to thank these  institutes for their  hospitality and support.

\section{The dynamical elliptic twisted  group algebra}

\subsection{Setup}

 We recall the notion of root datum following \cite[Exp. XXI, \S~1.1]{SGA3}. 
 A root datum of rank $n$ consists of two lattices $X^*$ and $X_*$ of rank $n$ with a perfect $\bbZ$-values pairing $\lag-,-\rag$,  non-empty finite subsets $\Phi\subset X^*$, and $\Phi^\vee\subset X_*$ with a bijection $\alpha\mapsto \alpha^\vee$.
For simplicity, we do not distinguish an abstract root datum and that of a connected, complex reductive group. If the group has maximal torus $T$, the lattices above are respectively $\bbX^*(T)$ and $\bbX_*(T)$, the group of characters and co-characters of $T$. 
In the present paper, we fix a set $\Sigma$ of simple roots $\{\alpha_1,\dots,\alpha_n\}$, with simple reflections $s_i:=s_{\al_i}$.  The Weyl group is denoted by $W$. For each $w\in W$, define $\Phi(w)=w\Phi^-\cap \Phi^+$. 
 Let $\rho$ be the half-sum of all positive roots.

Throughout this paper,  $E$ is an elliptic curve over $\bbC$ of the form
\[
E\cong \bbC/(\bbZ+\bbZ\tau)\cong \bbC^*/q^\bbZ, 
\]
where $\tau\in \bbC$ such that $\im \tau>0$, and $q=e^{2\pi i\tau}$.
We identify $E^\vee=\Pic^0(E)$ with $E$ itself, via the map $E\to E^\vee$, $\lambda\mapsto \{\lambda\}-\{0\}$.
In particular, on $E\times E$, there is a universal line bundle $\calP$, the Poincar\'e line bundle. 

Consider the origin $0\in E$. When the line bundle
  $\calO(\{0\})$ is lifted to the cover $\bbC^*$,  it becomes trivial. This lift is denoted by $\tilde{L}\to \bbC^*$,  and it comes with a lifted section  $\tilde{s}:\bbC^*\to \tilde{L}$.
The trivialization 
\[\xymatrix{
\tilde{L}\ar@/^/[dr]\ar[rr]^{\cong}_{\varphi}&&\bbC^*\times \bbC\ar[dl]\\
&\bbC^*\ar@/^/[ul]^{\tilde{s}}&
}\]
can be uniquely determined based on two conditions: $\varphi$ must commute with multiplication by 
$q^\bbZ$, and  the derivative of $\varphi\circ\tilde{s}$ should be equal to 1 at $1\in \bbC^*$ \cite[p.38]{Sie}.
Recall the Jacobi Theta function
\[\vartheta(u)=\frac{1}{2\pi i}(u^{1/2}-u^{-1/2})\prod_{s>0}(1-q^su)(1-q^su^{-1})\cdot \prod_{s>0}(1-q^s)^{-2}, \quad u\in \bbC^*.\]
Assume $|q|<1$, this series converges and defines a holomorphic function on a double cover on $\bbC^*$.
It vanishes of order 1 at $q^\bbZ$, is non-zero everywhere else, and its derivative at $1\in \bbC^*$ is equal to 1. 

Let us denote the coordinates of $E$ and $E^\vee$ as $z$ and $\lambda$, respectively. The following function
\[\frac{\vartheta(z-\lambda)}{\vartheta(z)}\]
becomes a rational section of 
 $\calP$.

\subsection{The (dynamical) elliptic twisted  group algebra}

 Define $\catA=X_*\otimes E$.  We have an isomorphism \begin{equation}\label{eqn:ab_dual}\catA^\vee\cong X^*\otimes E.\end{equation}
Each $\mu\in X^*$ defines a map $\chi_\mu:\catA\to E$. Denote $z_\mu=\chi_\mu(z), z\in \catA$. Dually, each $\mu^\vee\in X_*$ defines a map $\chi_{\mu^\vee}:\catA^\vee\to E$, and denote $\la_{\mu^\vee}=\chi_{\mu^\vee}(\la), \la\in \catA^\vee$. 
To make the  isomorphism \eqref{eqn:ab_dual} explicit, we consider a special point $\alpha\otimes t\in X^*\otimes E$, with $\alpha\in  X^*$ and $t\in E$. 
The point $\alpha$ defines a map $\chi_\alpha:\catA\to E$, and $t\in E=E^\vee$ defines a degree 0 line bundle $\calO(t)$ on $E$. 
Under the above isomorphism, the line bundle on $\catA$ corresponding to $\alpha\otimes t$ is $\chi_\alpha^*\calO(t)$. 
In what follows, we write $\alpha\otimes t$ simply as $t\alpha$. 
For each $z\in \catA, \la\in \catA^\vee$, the associated line bundle over $\catA$ is denoted by $\calO(\la)$. 

The Weyl group $W$ acts on both $\catA$ and $\catA^\vee$, and hence on the product $\catA\times\catA^\vee$, we have the action of product of two Weyl groups. To avoid confusion, we write the factor of $W$ acting on $\catA$ as $W$ and the factor acting on $\catA^\vee$ as $W^\dyn $. Similarly, for any element $w\in W$, we write the element in $W^\dyn $ as $w^\dyn $.

We define the category $\Coh^{W\times W^\dyn}(\catA\times \catA^\vee)$ of $W\times W^\dyn$-graded  coherent sheaves on $\catA\times \catA^\vee$. The objects are coherent sheaves with a direct-sum decomposition $\calF=\oplus_{w\in W,v\in W^\dyn}\calF_{w,v}$ with  $\calF_{w,v}$ a coherent sheaf, called the degree-$(w,v)$ piece of $\calF$. 
Morphisms in this category respect the grading. 
We then define a monoidal structure on this category. It suffices to work with each graded piece. We define   
\[
\calF_{w,v}\tensor \calF_{w',v'}':=\calF_{w,v}\otimes(w^{-1})^*(v^{-1})^\dynst\calF'_{w',v'},
\]
which is set to be of degree $(ww',vv')$. 
The object $\calO$ concentrated in degree 0 is the unit object.  
The verification of the monoidal axioms are straightforward.

On $\catA\times\catA^\vee$ there is a Poincar\'e line bundle $\calP$. See \cite[\S~9.3]{P} for its defining properties. Define \[\bbL:=\calP\otimes(\calO(\hbar\rho)\boxtimes\calO(-\hbar\rho^\vee)),\]
and define $\bbS_{w,v}=\bbL\otimes (w^{-1})^*(v^{-1})^{\dyn*}\bbL^{-1} $.
The object $\bbS_{W\times W^\dyn}:=\bigoplus_{w,v\in W}\bbS_{w,v}$ has the structure of a monad (algebra) object in the category $\Coh^{W\times W^\dyn}(\catA\times \catA^\vee)$ given by the identity map
\begin{align*}
\bbS_{w,v}\tensor\bbS_{x,y}&=\bbS_{w,v}\otimes (w^{-1})^*(v^{-1})^\dynst\bbS_{x,y}\\
&=\bbL\otimes (w^{-1})^*(v^{-1})^\dynst\bbL^{-1}\otimes (w^{-1})^*(v^{-1})^\dynst\bbL\otimes (w^{-1})^*(v^{-1})^\dynst (x^{-1})^*(y^{-1})^\dynst\bbL^{-1}\\
&\cong\bbL\otimes ((wx)^{-1})^*((vy)^{-1})^{\dyn*}\bbL^{-1}\\
&=\bbS_{wx,vy},
\end{align*}
called the {\it  (dynamical) elliptic  twisted group algebra}. 
We also write $\bbS$ for short. 

Notice also that the monoidal category $\Coh^{W\times W^\dyn}(\catA\times \catA^\vee)$ acts on $\Coh(\catA\times \catA^\vee)$. On graded pieces, the action is given by 
\[\calF_{w,v}\star\calG:=\calF_{w,v}\otimes (w^{-1})^*(v^{-1})^\dynst\calG.\] The object $\bbL$ in $\Coh(\catA\times \catA^\vee)$ with the identity map
\[\bbS_{w,v}\star \bbL\to \bbL\]
is a module over the algebra object $\bbS$. We refer interested readers to \cite{LZZ23} for more details of this monoidal category, the algebra object, as well as their applications in elliptic Schubert calculus. 

Similarly we have the monoidal categories $\Coh^W(\catA)$ and $\Coh^{W^\dyn }(\catA^\vee)$. Obviously, the twisted group algebra \begin{equation}\label{eqn:group_alg}
\calO_{\catA}\rtimes W:=\bigoplus_{w\in W}\calO_{\catA}\end{equation} is an algebra object in $\Coh^W(\catA)$. Similar for $\calO_{\catA^\vee}\rtimes W^\dyn $.

\section{Convolution categories}
\label{sec:conv}
In this section we set up the categorical framework.
\subsection{Convolution categories}
We define another monoidal category in which the elliptic Hecke algebra with the dynamical parameters is an algebra object. 

Let $\pi=\pi_1:\catA\to \catA/W$ be the categorical quotient. Similarly for $\pi_2: \catA^\vee\to \catA^\vee/W^\dyn $. By an abuse of notations we also  write $\pi_1:\catA\times\catA^\vee\to \catA/W\times\catA^\vee$, $\pi_2:\catA\times\catA^\vee\to \catA\times\catA^\vee/W^\dyn $, and $\pi^2:\catA\times\catA^\vee\to \catA/W\times\catA^\vee/W^\dyn $. By an abuse of notations we also write the projection $\pi_1:\catA\times_{\catA/W}\catA\to \catA/W$, $\pi_2:\catA^\vee\times_{\catA^\vee/W^\dyn }\catA^\vee\to \catA^\vee/W^\dyn $, and omit the subindices when causing no confusions. 

The category of coherent sheaves $\Coh(\catA\times_{\catA/W}\catA)$ has a monoidal structure defined as  follows. Notice that $\catA\times_{\catA/W}\catA\subseteq \catA\times\catA$.
Write the projection to the $(i,j)$-th factor as $p_{i,j}:\catA\times\catA\times\catA\to\catA\times\catA$ for $i,j=1,\dots,3$. Denote 
\[\Delta_{2}: \catA\times\catA\times\catA\to \catA\times\catA\times\catA\times\catA, ~(z_1,z_2,z_3)\mapsto (z_1,z_2,z_2,z_3).\] 
(Here the subscript 2 reminds us that $z_2$ is repeated.)
Define the product in $\Coh(\catA\times _{\catA/W}\catA)$:
\begin{equation}\label{eqn:conv}\calF\star\calF':=p_{13*}\Delta^*_{2}(\calF\boxtimes\calF').\end{equation}

\begin{lemma}\label{lem:monoidal_fun}
The functor $\phi:\Coh^{W}(\catA) \to \Coh(\catA\times_{\catA/W} \catA), 
\calF_{w}\mapsto \tilde w_*\calF_{w},$
is a monoidal functor. 
Here $\widetilde w_*$ is the pushforward along the map $\tilde w:\catA\to \catA\times_{\catA/W}\catA, z\mapsto (z,w^{-1}(z))$.
\end{lemma}

\begin{proof}
It suffices to work with individual graded pieces.  Given $\calF_{w}, \calG_{x}\in \Coh^{W}(\catA)$, we have by definition
\begin{align*}
\phi(\calF_{w}\star\calG_{x})=\phi(\calF_{w}\otimes(w^{-1})^* \calG_{x})=(\widetilde{wx})_*(\calF_{w}\otimes (w^{-1})^*\calG_{x})=(\widetilde{wx})_*\widetilde w^*(\calF_w\boxtimes \calG_x).
\end{align*}
On the other hand, 
\begin{align*}\phi(\calF_{w})\star\phi(\calG_{x})=\tilde w_*\calF_{w}\star \tilde x_*\calG_{x}=p_{13*}\Delta^*_{2}(\tilde w_*\calF_{w}\boxtimes \tilde x_*\calG_{x}).
\end{align*}
We show that they coincide.

Denote $\catA^i$ as the product of $i$ copies of $\catA$.  We have the following commutative diagram, where the right square is cartesian:
\[
\xymatrix{&\catA\ar[r]^-{\widetilde w}\ar[ld]_-{\widetilde{wx}}\ar[d]^-p& \catA^2\ar[d]^-{\tilde w\times \tilde x}\\
\catA^2&\catA^3\ar[r]^-{\Delta_{2}}\ar[l]^-{p_{13}}& \catA^4}.
\]
Here $p(z)=(z,w^{-1}(z), x^{-1}w^{-1}(z))$, so $(\widetilde{wx})_*=(p_{13})_*p_*$ and $p_*\tilde x^*=\Delta_{2}^*(\tilde w\times \tilde x)_*$. 
Therefore, 
\begin{align*}
(\widetilde{wx})_*\widetilde w^*(\calF_w\boxtimes \calG_x)&=(p_{13})_*p_*\widetilde w^*(\calF_w\boxtimes \calG_x)
&=(p_{13})_*\Delta_{2}^*(\tilde w\times \tilde x)_*(\calF_w\boxtimes \calG_x)
&=(p_{13})_*\Delta_{2}^*(\tilde w_*\calF_w\boxtimes \tilde x_*\calG_x). 
\end{align*}
The proof is finished. 
\end{proof}

The followings are similar and hence we omit the proofs. 
\begin{lemma}\label{lem:mon_fun_others}
We have the following monoidal functors, where the monoidal structures on the domains and targets are similar to those in Lemma~\ref{lem:monoidal_fun}.
\begin{enumerate}
\item $\Coh^{W\times W^\dyn}(\catA\times \catA^\vee) \to \Coh^{W^\dyn}(\catA\times_{\catA/W} \catA\times \catA^\vee)$, $
\calF_{w,v}\mapsto \tilde w_*\calF_{w,v}$. 
\item $\Coh^{ W^\dyn}(\catA\times_{\catA/W}\catA\times \catA^\vee) \to \Coh(\catA\times_{\catA/W} \catA\times \catA^\vee\times_{\catA^\vee/W^\dyn }\catA^{\vee})$, $\calF_{v}\mapsto \tilde v^\dyn_*\calF_{v}$. 
\end{enumerate}
\end{lemma}
Without causing any confusion, we denote all of these functors (as well as their localizations explained below) by $\phi$.

\subsection{Pushforward}
\begin{lemma}\label{lem:lax}
The pushforward functor $\pi_*: \Coh(\catA\times_{\catA/W}\catA)\to \Coh(\catA/W)$ is lax monoidal, with  the target  endowed with the usual tensor structure of coherent sheaves. 
Similar for the followings:
\begin{enumerate}
\item  $ \Coh(\catA\times_{\catA/W}\catA\times\catA^\vee\times_{\catA^\vee/W^\dyn }\catA^\vee)\to \Coh(\catA/W\times\catA^\vee\times_{\catA^\vee/W^\dyn }\catA^\vee)$;
\item $ \Coh^{W^\dyn }(\catA\times_{\catA/W}\catA\times\catA^\vee)\to \Coh^{W^\dyn }(\catA/W\times\catA^\vee)$. 
\end{enumerate}
\end{lemma}
\begin{proof}
We only show the conclusion for the functor $ \Coh(\catA\times_{\catA/W}\catA)\to \Coh(\catA/W)$. The others are similar. 

We have the following commutative diagram where the right square is Cartesian:
\[\xymatrix{
\catA\times_{\catA/W}\catA\times_{\catA/W}\catA\ar[r]^-r\ar@/^20pt/[rr]^-{\Delta_{23}}\ar[d]_-{p_{13}}&X\ar[d]^-b\ar[r]^-a &\catA\times_{\catA/W}\catA\times\catA\times_{\catA/W}\catA\ar[d]_{\pi\times\pi}\\
\catA\times_{\catA/W}\catA\ar[r]_-{\pi}&
\catA/W\ar[r]_-{\Delta}&\catA/W\times \catA/W.
}\]
The map $r$ is induced by the universal property of the cartesian square. 
The adjunction $\id\to r_*r^*$ induces  a sequence of natural transformations 
\[
\Delta^*(\pi\times \pi)_*=b_*a^*\to b_*r_*r^*a^*=\pi_*(p_{13})_*\Delta_{23}^*.
\]
It is then straightforward to verify that this natural transformation satisfies the Lax axioms.
\end{proof}
For the sake of clarity, we spell out the map 
\[ \pi_* \phi(\calF_w)\otimes\pi_* \phi(\calG_v)\to \pi_*\phi(\calF_w\star\calG_v)\]
for $\pi_*\phi:\Coh^W(\catA)\to \Coh(\catA/W)$. Given $\calF_w$ and $\calG_v$ of degree $w,v$ in $\Coh^W(\catA)$, we have the composition
\[
\pi_*\phi\calF_w\otimes \pi_*\phi\calG_v\to \pi_*(\phi(\calF_w)\star \phi(\calG_v))=\pi_*\phi(\calF_w\star\calG_v)=\pi_*\phi(\calF_w\otimes(w^{-1})^* \calG_v). 
\]
For this reason, with $\calO_{\catA}\rtimes W$ defined in \eqref{eqn:group_alg},  $\pi_*(\calO_{\catA}\rtimes W)$ in $\Coh(\catA/W)$ is the usual twisted group algebra, the section of which on an open subset $U$ is written as \[\sum_{w\in W}a_w\delta_w,\]
where $\de_w$ is to indicate the degree. The  twisted product is written as 
\[
a_w\de_w b_v\de_v=a_w\cdot {}^w{b_v}\de_{wv}.
\]
Here for an element $w\in W$ and a function $f$ on $U\subseteq \catA$, we denote the image of $f$ under the map $w^{-1}:\calO(U)\to \calO(w^{-1}U)$ by ${}^wf$.

\subsection{Localization functors}\label{subsec:loc}
For each $\alpha\in \Phi$,
let $D_\al$ be the hyperplane of $\catA$ defined by $\al$, i.e., $D_\al=\ker \chi_\al$, and let $D_{ \hbar, \al}$ be the divisor of $\catA$ defined by $\chi_\al-\hbar$.  We also view $D_\al$ and $D_{\hbar,\al}$ as divisors of $\catA\times \catA^\vee$.  Let $J_{\hbar,\al}$ be the sheaf of ideals of $D_{\hbar, \al}$.
Let $\catA^\reg\subseteq\catA$ be the complement of  $\cup_{\al\in\Phi} (D_\al\cup D_{ \hbar, \al})$ and denote the embedding $j:\catA^\reg\to \catA$.  
Similarly, define $D_{\al^\vee}$ and $D_{\hbar,\al^\vee}$ as divisors of $\catA^\vee$, and also view them as divisors of $\catA\times \catA^\vee$. Define $(\catA^\vee)^\reg$ similarly, and denote $j^\vee:(\catA^\vee)^\reg\to \catA^\vee$. The open subsets $\catA^\reg$ and $(\catA^\vee)^\reg$ are invariant respectively under $W$ and $W^\dyn$.  Hence, the categories $\Coh^W(\catA^\reg)$ and $\Coh(\catA^\reg\times_{\catA/W} \catA^\reg)$ have induced monoidal structures, so that the localization functors
\[\Coh^W(\catA)\to \Coh^W(\catA^\reg)\]  and \[\Coh(\catA\times_{\catA/W} \catA)\to \Coh(\catA^\reg\times_{\catA/W} \catA^\reg)\] are both monoidal functors. 
Similar for $\catA^\vee$, and other categories involved in Lemma~\ref{lem:mon_fun_others}.

We have the following commutative diagram of monoidal categories
\[\xymatrix{\Coh^{W}(\catA^{\reg})\ar[r]_-{\cong} & \Coh^{}(\catA^\reg\times_{\catA/W}\catA^\reg)\\
\Coh^{W}(\catA)\ar[u]\ar[r]^-\phi & \Coh^{}(\catA\times_{\catA/W} \catA )\ar[u]. },
\]
where the vertical functors are localization functors. Notice that the top row is an equivalence of categories because the action of $W$ on $\catA^{\reg}$ is free.
Similarly, we have the following commutative diagram of monoidal categories with top row being  an equivalence of categories
\[
\xymatrix{\Coh^{W\times W^\dyn}(\catA^{\reg}\times (\catA^\vee)^\reg)\ar[r]_-{\cong} & \Coh^{}(\catA^\reg\times_{\catA/W}\catA^\reg\times( \catA^\vee)^{\reg}\times_{\catA^\vee/W^\dyn } (\catA^\vee)^{\reg})\\
\Coh^{W\times W^\dyn}(\catA\times \catA^\vee)\ar[u]\ar[r]^-\phi & \Coh^{}(\catA\times_{\catA/W} \catA\times \catA^\vee\times_{\catA^\vee/W^\dyn }\catA^\vee )\ar[u]. }.
\]

We also  have a natural isomorphism 
\[\calO_{\catA^\reg}\rtimes W\cong \phi(\calO_{\catA^\reg\times_{\catA/W} \catA^\reg}).\]
The residue conditions \cite{GKV97} define a subalgebra object of $j_*\calO_{\catA^\reg}\rtimes W$
in the category of quasi-coherent sheaves
$Q\Coh(\catA\times_{\catA/W} \catA)$, which we denote by  $\calH$. 
We refer $\calH$ as the GKV algebra, an algebra object in $\Coh(\catA\times_{\catA/W} \catA)$. This construction is reviewed below in Remark~\ref{rem:GKV}. 

\section{Residue definition of the dynamical Hecke algebra}
We define the elliptic Hecke algebra with dynamical parameters.
This algebra is not only the key ingredient in construction of the Fourier-Mukai functor, but also interesting on its own. The images of this algebra under various localization functors have already appeared in \cite{ZZ15, LZZ23}. However, for the purpose of representation theory, it is necessary to find a suitable integral structure, which we carry out in this section, following   \cite[Definition 1.3]{GKV97}.
\label{sec:res}
\subsection{The residues of rational sections of line bundles}
Let $U\subseteq \catA\times\catA^\vee$ be an open subset, with $f$ a rational section of $\calL$ on $U$, with a pole along $D_\alpha$   of order at most 1.  
The residue of $f$ along $D_\alpha$ is denoted by 
\[
\Res_\al (f)=(\vartheta(z_\al)f)|_{D_\al}.
\]
In words, taking residue  is the same as evaluating $\vartheta(z_\al) f$ at $z_\al=0$. In particular, the result is a rational section of $(\calL\otimes \calO(-D_\al))|_{D_\al}$ on $D_\alpha\cap U$.
If $a,b$ are rational sections of line bundles $\calL_1,\calL_2$, together with a canonical isomorphism $\calL_1|_{D_\al}\cong \calL_2|_{D_\al}$, then $a|_{D_\al}+b|_{D_\al}$ is defined via this isomorphism. In this case, the following definition makes sense:
\[
\Res_\al(a+b):=\Res_\al(a)+\Res_\al (b). 
\]  
 
\begin{lemma}\label{lem:s_al}
\begin{enumerate}
\item We have canonical  isomorphisms $s_\al^*\bbS_{s_\al,e}\cong\bbS_{s_\al,e}^{-1}$ and  $\bbS_{w,v}|_{D_\al}\cong \bbS_{s_\al w, v}|_{D_\al}$. 
\item If $f_\al$ is a rational section of $\bbS_{s_\al,e}$, and $a$ is a rational section of $\bbS_{s_\al w,v}$, then ${}^{s_\al}a\cdot  f_\al$ is a rational section of $\bbS_{w,v}$. 
\end{enumerate}
\end{lemma}
\begin{proof}
(1).The first one follows from the definition $\bbS_{s_\al,e}=\bbL\otimes s_\al^*\bbL^{-1}$. For the second one, recall that  \[
\bbS_{s_\al w, v}\cong\bbS_{s_\al, e}\otimes s_\al^*\bbS_{w, v}.
\]
Note that $s_\al$ fixes $D_\al$, so 
\[
\bbS_{s_\al, e}|_{D_\al}=(\bbL\otimes s_\al^*\bbL^{-1})|_{D_\al}=\bbL|_{D_\al}\otimes (s_\al^*\bbL^{-1})|_{D_\al}\cong\bbL|_{D_\al}\otimes \bbL^{-1}|_{D_\al}=\calO|_{D_\al}.
\]
Similarly, $(s_\al^*\bbS_{w, v})|_{D_\al}\cong\bbS_{w, v}|_{D_\al}$. Therefore, $\bbS_{w,v}|_{D_\al}\cong \bbS_{s_\al w, v}|_{D_\al}$.

(2). The isomorphism of (1) gives 
\[
s_\al^*\bbS_{s_\al w,v}\cong s_\al^*\bbS_{s_\al,e}\otimes \bbS_{w,v}\cong \bbS_{s_\al,e}^{-1}\otimes \bbS_{w,v}.
\]
The conclusion then follows.
\end{proof}

Consequently,
if $f$ is a rational section of $\bbS_{w,v}$  and $g$ is a rational section of  $\bbS_{s_\al w,v}$, 
$\Res_\al(f+g)=\Res_\al(f)+\Res_\al(g)$ makes sense.
\subsection{The elliptic Hecke algebra with dynamical parameters}
Consider  $j^2:=j\times j^\vee:\catA^\reg\times\catA^{\vee\reg}\to \catA\times \catA^\vee$, and denote $\pi^2:\catA\times \catA^\vee\to \catA/W\times \catA^\vee/W^\dyn$.  

As in \cite[\S~2.7]{LZZ23}, we write a rational section of $\bbS$ as 
\begin{equation}\label{eqn:sum}
\sum_{w,v\in W}a_{w,v}\de_w\de_v^\dyn,
\end{equation}
where $a_{w,v}$ is a rational section of $\bbS_{w, v}$ and $\de_w,\de_v^\dyn$ are to denote its degree. We also write $\de_e$ as $1$, or $\id$ when acting on a module. 

\begin{definition}\label{def:res}
Let $\tilde \calH^\dyn$ be the subsheaf of $j^2_*(j^2)^*\bbS$, so that on any open subset $U\subseteq \catA\times \catA^\vee$,  $\tilde \calH^\dyn(U)$ consists of  $\sum_{w,v}a_{w,v}\de_w\de_v^\dyn\in j^2_*(j^2)^*\bbS(U)$ satisfying  conditions i), ii), and ii') below. Let $\calH^\dyn$ be the subsheaf of $\tilde \calH$ whose sections on $U$ satisfies the  condition iii) below. 
\begin{enumerate}
\item[i)]  The possible poles of $a_{w,v}$ are at $D_\al$ or $D_{\al^\vee}$, of orders at most 1, for finitely many $\al\in \Phi_+$. 
\item[ii)] $\Res_\al(a_{w,v}+a_{s_\al w, v})=0.$
\item[ii')] $\Res_{\al^\vee}(a_{w,v}+a_{w,s_\al v})=0$.
\item[iii)] For any $\al\in \Phi(w)$, as a rational section of $\bbS_{w,v}\otimes \tilde\calO(-\hbar\al)$,  $a_{w,v}\frac{\vartheta(z_\al)}{\vartheta(\hbar-z_\al)}$ is regular at $D_{\hbar,\al}$.
\end{enumerate}
\end{definition}

 The conditions on an expression \eqref{eqn:sum} defining $\tilde \calH^\dyn$ and $ \calH^\dyn$  are invariant under left multiplication actions of functions on $\catA\times\catA^\vee$. Hence  $\tilde\calH^\dyn$ and $\calH^\dyn$ are coherent sheaves on $\catA\times\catA^\vee$. Below we will consider properties of $\tilde\calH^\dyn$ and $\calH^\dyn$ are coherent sheaves. 
 Here we notice that  the projection $\catA\times\catA^\vee\to \catA/W\times\catA^\vee/W^\dyn$ is a finite morphism. 
 Recall that for a coherent sheaf on $\catA/W\times \catA^\vee/W^\dyn$, the structure of a coherent sheaf on $\catA\times\catA^\vee$ is equivalent to a module structure over $\pi^2_*\calO$.  Pushing-forward $\tilde\calH^\dyn$ and $\calH^\dyn$ makes them into  coherent sheaves  on $\catA/W\times \catA^\vee/W^\dyn$ with left multiplications  of $\pi^2_*\calO$. 
Similarly, the right multiplication of $\pi^2_*\calO$ on $\pi^2_*\tilde\calH^\dyn$ and $\pi^2_*\calH^\dyn$ defines another structure of $\tilde \calH^\dyn$ and $ \calH^\dyn$ as coherent sheave on $\catA\times\catA^\vee$. The left and right actions of $\pi^2_*\calO$ agree when restricted to $(\pi^2_*\calO)^{W\times W^\dyn}$, hence $\tilde \calH^\dyn$ and $ \calH^\dyn$  are
coherent sheaves on $\catA\times_{\catA/W} \catA\times \catA^\vee\times_{\catA^\vee/W^\dyn }\catA^\vee\to \catA\times \catA^\vee$.

\begin{remark}\label{rem:GKV}
Let $\pi:\catA\to \catA/W$.  Recall the algebra object 
$\calO_{\catA}\rtimes W$ in \eqref{eqn:group_alg}, as well as $j_*j^*\calO_{\catA}\rtimes W$.
Let  $\tilde\calH$ (resp. $\calH$) to be the subsheaf of $j_*j^*\calO_{\catA}\rtimes W$, so that 
on any open subset $U\subseteq \catA$,  $\tilde\calH(U)$ (resp. $\calH(U)$) consists of  $\sum_wa_w\de_w\in j_*j^*\calO_{\catA}\rtimes W(U)$ satisfying  conditions i) and ii)  (resp. i), ii), iii)). 
Here condition iii) in this setting yields: for any $\al\in \Phi(w)$, $a_w$ vanishes at $D_{\hbar, \al}$. 
The theorem \cite[Theorem 1.4]{GKV97} shows that $\tilde\calH$ and $\calH$ are both algebra objects, and since $\pi_*$ is lax monoidal, so $
\pi_*\tilde\calH$ and $\pi_*\calH$ are coherent sheaves of algebras (see, e.g., \cite[Proposition~3.29]{AM10}). 
Moreover,  \cite[Theorem 2.2, Theorem 4.4]{GKV97} give,
\[\pi_*\tilde \calH\cong \sEnd (\pi_*\calO_{\catA}), \quad \pi_*\calH=\{g\in \tilde\calH\mid g(\pi_*J_{\hbar})\subset \pi_*J_{\hbar}\}. \]
Here $J_{\hbar}$ is the sheaf of ideals $\prod_{\al}J_{\hbar, \al}$ corresponding to the divisor $\cup_\al D_{\hbar, \al}$.  
\end{remark}
Similarly, starting with the Langlands dual root datum, we have  $\tilde\calH^L$ and $\calH^L$.
The rest of this section establishes analogues of these statements in the presence of the dynamical parameters. 

\begin{theorem}\label{thm:mult}
The subsheaves $\tilde\calH^\dyn$ and $\calH^\dyn$ are  subsheaves of algebras of $\bbS$. Moreover,  $\pi_*^2\tilde\calH^\dyn$ and $\pi_*^2\calH^\dyn$ are subsheaves of algebras of $\pi_*^2\bbS$. 
\end{theorem}
\begin{proof} 
If suffices to prove the lemma by proving the conclusion for the stalks at the $W\times W^\dyn$-orbit of $p=(z,\la)\in \catA\times \catA^\vee$. 
We fix a local trivialization of $\bbL_p\cong (\calO_{\catA\times \catA^\vee})_{p}$, which induces a trivialization at each point of the orbit. Denote $\bar p=\pi^2(p)$. 
Notice that with the fixed trivialization, we have an isomorphism \[(\pi_*^2\bbS)_{\overline{p}}\cong (\pi_*^2\calO)_{\overline{p}}\rtimes (W\times W^\dyn),\] and hence $(\pi_*^2\tilde\calH^\dyn)_{\overline{p}}\cong \left(\pi_*^2(\tilde\calH\boxtimes \tilde\calH^L)\right)_{\overline{p}}$ and $(\pi_*^2\calH^\dyn)_{\overline{p}}\cong \left(\pi_*^2(\calH\boxtimes \tilde\calH^L)\right)_{\overline{p}}$.
Now it follows from  \cite[Theorem 1.4]{GKV97} that the first two conditions and the first three conditions all define  sheaves of subalgebras. 

By Lemma \ref{lem:lax}, the functor $\pi_*^2$  is lax monoidal, so $\pi_*^2 \tilde\calH^\dyn$ and $\pi_*^2\calH^\dyn$ are both sheaves of algebras. 
\end{proof}

As coherent sheaves on $\catA\times\catA^\vee$, $\tilde\calH^\dyn$ and $\calH^\dyn$ have filtrations by the Bruhat order of $W$. 
In terms of sections, for any $x\in W$, the subsheaf $\calF_{\leq x}\tilde\calH^\dyn\subseteq \tilde\calH^\dyn$ has sections on $U\subseteq \catA/W\times\catA^\vee/W$ consisting of linear combinations $\sum_{w,v}a_{w,v}\de_w\de_v^\dyn\in j^2_*(j^2)^*\bbS(U)$ with $w\leq x$. Similar for $\calF_{\leq x}\calH^\dyn$.
  
\begin{lemma}\label{lem:H_basic}
\begin{enumerate}
\item  The sheaves of algebras $\pi_*^2\tilde\calH^\dyn, \pi_*^2\calH^\dyn$ are generated by  $\calF_{\leq s_\alpha}$ with $\alpha$ simple roots. 
\item The sheaves $\tilde\calH^\dyn, \calH^\dyn$ are locally free.
\item The sheaves  $\pi_*^2\tilde\calH^\dyn, \pi_*^2\calH^\dyn, \pi_*^2\bbL$ are reflexive.
\end{enumerate}
\end{lemma}

Here (1) means for each point $p\in \catA/W\times\catA^\vee/W$, the stalks of  $\pi_*^2\tilde\calH^\dyn, \pi_*^2\calH^\dyn$ at $p$ can be written as products of
elements of the form $a\de_\al^\dyn$ and $b\de_\al^\dyn+c\de_\al \de_\al^\dyn$. 
Notice also that (2) is equivalent to the associated graded sheaves $\gr\tilde\calH^\dyn $ and $\gr \calH^\dyn$ are locally free, since extensions of locally free sheaves are locally free. 

Note that $\catA/W\times \catA^\vee/W^\dyn$ is a weighted projective space, which may not be smooth. Assuming Lemma~\ref{lem:H_basic}(2), the sheaves $\pi_*^2\tilde\calH^\dyn, \pi_*^2\calH^\dyn$ are locally free if and only if  $\catA/W\times \catA^\vee/W^\dyn$ is smooth. (This fact is not used in the present paper. We give a proof in Appendix~\ref{sec:App} for the convenience of the readers).
Nevertheless, Lemma~\ref{lem:H_basic}(3) asserts that the sheaves $\pi_*^2\tilde\calH^\dyn, \pi_*^2\calH^\dyn$ are reflexive, a fact which suffices in order to prove Theorem \ref{thm:poly}. 
For the convenience of the readers, we collect some basic facts regarding reflexive sheaves in Appendix~\ref{sec:App}. Furthermore,  Lemma~\ref{lem:MCM}  is used in the proof of Lemma~\ref{lem:H_basic}. 

\begin{proof}[Proof of Lemma~\ref{lem:H_basic}]
Similar as in the proof of Theorem~\ref{thm:mult}, (1) and (2) are local questions. With the help of a local trivialization of $\bbL$, 
they follow from \cite[Lemma~2.8(ii)]{GKV97} and \cite[Lemma~2.8(i)]{GKV97} respectively (see also, \cite[Lemma~4.6]{ZZ15}).

For (3), 
notice the map $\pi^2:\catA\times \catA^\vee\to \catA/W\times \catA^\vee/W^\dyn$ is a quotient by a finite group.   Lemma \ref{lem:MCM} below  shows that $\pi_*^2\tilde\calH^\dyn, \pi_*^2\calH^\dyn, \pi_*^2\bbL$ are maximal Cohen-Macaulay, hence reflexive (Lemma~\ref{lem:ref}(2)). 
\end{proof}

For later use, we have the following. 
\begin{lemma}
\label{lem:flat}
The vector bundle $\calH$ on $\catA$ is flat. 
\end{lemma}
\begin{proof}
A iterated extension of finitely many flat bundles is flat. Hence it suffices to show each graded piece in $\gr\calH$ is flat. This can be done by induction on the Bruhat order, and the observation that the line bundle $\calO(\{0\}-\{\hbar\})$ is flat on an elliptic curve.
\end{proof}

\subsection{The polynomial representation}
Let  $U\subset \catA/W\times \catA^\vee/W^\dyn$ be an open subset. Let $g:=\sum_{w,v}a_{w,v}\de_w\de_v^\dyn$ be a  section of $\pi_*^2\tilde\calH(U)$, and $f$ be a  section of $\pi_*^2\bbL(U)$,  then  the action of $g$ on $f$ is defined as 
\[
g( f)=\sum_{w,v}a_{w,v}\cdot {}^{wv^\dyn}f.
\]
Then similar as in \cite[Lemma 3.1]{GKV97},  the residue conditions i) and ii) guarantee that the poles of $a_{w,v}$ will be cancelled in $g(f)$. Moreover, from the isomorphism
\[
\bbS_{w,v}\otimes (w^{-1})^*(v^{-1})^\dynst \bbL=\bbL,
\]
we see that $g( f)$ is a  section of $\pi_*^2\bbL(U)$. So we obtain a well defined map of sheaves
\[
\psi:\pi_*^2\tilde\calH^\dyn\to \sEnd(\pi_*^2\bbL).
\]
Observe that after applying $(j^2)^*$, the sheaves $\tilde\calH^\dyn$, $\calH^\dyn$, and $\bbS$ coincide, and the  map $\psi$ becomes an isomorphism. 
\begin{theorem}\label{thm:poly}
The map $\psi$ is an  isomorphism of sheaves
\[
\pi^2_*\tilde\calH^\dyn\to \sEnd(\pi^2_*\bbL).
\]
\end{theorem}
\begin{proof}
We have the following commutative diagram:
\[
\xymatrix{\pi_*^2\tilde\calH^\dyn\ar[r]^-\psi \ar[d]& \sEnd(\pi_*^2\bbL)\ar[d]_l\\
\cap _{\codim \bar p=1}(\pi_*^2\tilde\calH)_{\bar p}^\dyn\ar[r] ^-{\cap \psi_{\bar p}}& \cap_{ \codim \bar p=1} \sEnd(\pi_*^2\bbL)_{\bar p}}.
\]
Here the intersections are taken among all codimension-1 points $\bar p$ of the scheme $\catA/W\times \catA^\vee/W^\dyn$.
The sheaf $\pi_*^2\bbL$ is reflexive (Lemma~\ref{lem:H_basic}(3)), hence torsion-free. Consequently, $\sEnd(\pi_*^2\bbL)$ is torsion-free and hence the right vertical map $l$ is injective.
Since $\pi_*^2\tilde\calH^\dyn $ is reflexive (Lemma~\ref{lem:H_basic}(3)), the left vertical map  is an isomorphism  (Lemma~\ref{lem:ref}(1)). 
It follows from \cite[Lemma 3.5]{GKV97} that $\psi_{\bar p}$ is surjective. Therefore,  the left-bottom composition is surjective. 
Consequently, $l$ is surjective and hence an isomorphism. 
Again by Lemma \ref{lem:H_basic}, $\psi$ is a map between torsion-free sheaves. From the observation above, $(j^2)^*\psi$ is injective. Therefore, $\psi$ is injective, and hence  an isomorphism. 
\end{proof}

The sheaf 
\[
J_{\hbar, \bbL}:=(\otimes_{\calO}J_{\hbar,\al})\bbL, 
\]
is a subsheaf of modules of $\bbL$. It follows from \cite[Lemma 3.1]{GKV97} that condition iii) implies that $\psi(\pi_*^2\calH^\dyn)$ preserves $\pi_*^2J_{\hbar, \bbL}$. 
Moreover,  it follows from \cite[Theorem 2.2]{GKV97},
\begin{equation}\label{eq:end}\pi_*^2\calH^\dyn=\{g\in \tilde\calH^\dyn\mid g(\pi_*^2J_{\hbar, \bbL})\subset \pi_*^2J_{\hbar, \bbL}\}.
\end{equation}

\subsection{ The Demazure-Lusztig operators with dynamical parameters}
\label{subsec:DL}
For each $\alpha\in \Sigma$, we consider the dynamical DL operator  $T_\al=T^{z,\la}_\al$ (see \cite{LZZ23}): \[T_\al=\frac{\theta(\hbar)\theta(z_\al-\la_{\al^\vee})}{\theta(z_\al)\theta(\hbar-\la_{\al^\vee})}\de_\al^\dyn+\frac{\theta(\la_{\al^\vee})\theta(\hbar-z_\al)}{\theta(z_\al)\theta(\hbar-\la_{\al^\vee})}\de_\al\de_\al^\dyn.\]
It  is a rational section of $\bbS$ (\cite[Theorem 6.3]{LZZ23}), and it is easy to see that $T_\al$ is a global section of $\calH^{\dyn}|_{\catA\times(\catA^\vee)^{\reg}}$. 
Indeed, Condition i) is satisfied obviously. For Condition ii), we have 
\[
\Res_\al(\frac{\theta(\hbar)\theta(z_\al-\la_{\al^\vee})}{\theta(z_\al)\theta(\hbar-\la_{\al^\vee})}+\frac{\theta(\la_{\al^\vee})\theta(\hbar-z_\al)}{\theta(z_\al)\theta(\hbar-\la_{\al^\vee})})=(\frac{\theta(\hbar)\theta(z_\al-\la_{\al^\vee})}{\theta(\hbar-\la_{\al^\vee})}+\frac{\theta(\la_{\al^\vee})\theta(\hbar-z_\al)}{\theta(\hbar-\la_{\al^\vee})})|_{z_\al=0}=0.
\] 
Condition ii') can be checked similarly. 
Lastly, condition iii)  is vacuous on the coefficient of $\de_\al^\dyn$. The coefficient of $\de_\al^\dyn\de_\al$ multiply with $\frac{\theta(z_\al)}{\theta(\hbar-z_\al)}$ is equal to $\frac{\theta(\la_{\al^\vee})}{\theta(\hbar-\la_{\al^\vee})}$, which is regular at $D_{\hbar,\al}$.

 It is shown by Rim\'anyi and Weber that they satisfy the Weyl group relations, that is,  the braid relations and also $T_\al^2=1$. (See \cite[Proposition 4.11]{ZZ15} for a statement of the present form.)
Therefore, $T_w$ for any reduced word decomposition of $w\in W$ is well defined and is  a global section of $\calH^{\dyn}|_{\catA\times(\catA^\vee)^{\reg}}$.

\section{Fourier-Mukai transform in steps}
In this section we prove the main results of the present paper summarized in the Introduction. 
\subsection{Module categories}

The monoidal category $\Coh(\catA\times_{\catA/W} \catA)$ acts on the category $\Coh(\catA)$  via convolution, similar to \eqref{eqn:conv}.
 For clarity, let $p_i:\catA\times\catA\to \catA$ be the $i$th projection, $i=1,2$, and let $\Delta_{2}:\catA\times\catA\to \catA\times\catA\times\catA, (z_1,z_2)\mapsto (z_1,z_2,z_2)$. For a coherent sheaf $\calF$ on $\catA$ and $\calG$ on $\catA\times_{\catA/W} \catA$, define \[\calG\star\calF:=p_{1*}\Delta^*_{2}\calG\boxtimes\calF.\]

Recall that $\calH$ is an algebra object in $\Coh(\catA\times_{\catA/W} \catA)$. We define $\Mod \calH$ to be the category of $\calH$-modules in $\Coh(\catA)$. 
Notice this notion is equivalent to the category of coherent sheaves of $\pi_*\calH$ modules on $\catA/W$. Indeed, using the lax property of $\pi_*$ similar to Lemma~\ref{lem:lax}, any objects in the former category gives rise to an object in the later via $\pi_*$. Conversely, any $\pi_*\calH$-module has the action of the subsheaf of algebras $\pi_*\calO_{\catA}\inj \pi_*\calH$, which endows it the structure of a coherent sheaf on $\catA$ \cite{Hart}.

Define $\Mod_{fin} \calH$ to be the full subcategory of $\Mod \calH$, whose underlying coherent sheaf has zero-dimensional support in $\catA$. Define $\Mod_{fl} \calH$ to be the full subcategory of $\Mod \calH$, whose underlying coherent sheaf is a homogeneous vector bundle on $\catA$.

Similarly, the monoidal category $\Coh(\catA\times_{\catA/W} \catA\times\catA^\vee\times_{\catA^\vee/W^\dyn }\catA^\vee)$ acts on $\Coh(\catA\times\catA^\vee)$  via convolution. 
We define $\Mod \calH^{\dyn}$ to be  the category of $\calH^{\dyn}$-modules in $\Coh(\catA\times\catA^\vee)$.

\subsection{From GKV algebra to dynamical Hecke algebra}\label{subsec:GKV-dyn} 
Let $p:\catA\times\catA^\vee\to \catA$ and $q:\catA\times\catA^\vee\to \catA^\vee$ be the two projections. 
Let us consider the functor \[{(\bbL\otimes -)\circ p^*}:  \Coh(\catA)\to \Coh(\catA\times\catA^\vee).\]
Notice that the image of the functor ${(\bbL\otimes -)\circ p^*}$ only  consists  of coherent sheaves on $\catA\times\catA^\vee$ which are flat over $\catA^\vee$.

Let $\Mod\calH\to \Coh(\catA)$ and $\Mod\calH^{\dyn}\to \Coh(\catA\times\catA^\vee)$ be the forgetful functors. 
\begin{theorem}\label{thm:pullback}
There is a functor $\Mod\calH\to \Mod\calH^{\dyn}$ making the following diagram commutative
\[\xymatrix{
\Mod\calH\ar[rr]\ar[d]& &\Mod\calH^{\dyn}\ar[d]\\
 \Coh(\catA)\ar[rr]_-{(\bbL\otimes -)\circ p^*}& & \Coh(\catA\times\catA^\vee).
}\]
\end{theorem}

In words, let $\calF$ in $\Coh(\catA)$ be a module over $\calH$, then, $\bbL\otimes p^*\calF$ has the natural structure of a module over $\calH^\dyn $. The proof is carried out in the rest of \S~\ref{subsec:GKV-dyn}.
Roughly, we show this in two step: 1. construct a coproduct-like algebra homomorphism $\Psi:\calH^\dyn \to (\calH^\dyn+\bbS)\otimes p^*\calH$; 2. we show that this map induces an action of $\calH^\dyn$ on $\pi_*^2(\bbL\otimes p^*\calF)$. 
Cares need to be taken as the target of $\Psi$ does not have an obvious algebra structure.

\subsubsection{Step 1}
First, recall that by Lemma~\ref{lem:monoidal_fun}, $\bbS$ is a coherent sheaf on $\catA\times_{\catA/W} \catA\times \catA^\vee\times_{\catA^\vee/W^\dyn }\catA^{\vee}$ and $\calO\rtimes W$ is a coherent sheaf on $\catA\times_{\catA/W} \catA$.
Similar to \S~\ref{sec:conv}, the category \[\Coh\big( (\catA\times_{\catA/W} \catA\times \catA^\vee\times_{\catA^\vee/W^\dyn }\catA^{\vee})\times_{\catA/W}(\catA\times_{\catA/W}\catA)\big)\] has a convolution monoidal structure, in which 
$
\bbS\boxtimes (\calO\rtimes W)$
is an algebra object. 
Let
$\Delta_\catA:  \catA\times_{\catA/W} \catA\times \catA^\vee\times_{\catA^\vee/W^\dyn }\catA^{\vee}\to  (\catA\times_{\catA/W} \catA\times \catA^\vee\times_{\catA^\vee/W^\dyn }\catA^{\vee})\times_{\catA/W}\catA$
be the map $(a,b,c,d)\mapsto (a,b,c,d, a)$.
It is well known that $\Delta_\catA^*(\bbS\boxtimes (\calO\rtimes W))=\bbS\otimes p^*(\calO\rtimes W)$, and  
$\Delta_\catA^*(\calH^\dyn \boxtimes \calH)=\calH^\dyn \otimes p^*\calH $.

We consider  the projection to the first  factor  \[\proj:(\catA\times_{\catA/W} \catA\times \catA^\vee\times_{\catA^\vee/W^\dyn }\catA^{\vee})\times_{\catA/W}(\catA\times_{\catA/W}\catA)\to \catA\times_{\catA/W} \catA\times \catA^\vee\times_{\catA^\vee/W^\dyn }\catA^{\vee}.\]  Since $\Delta_\catA$ is a section of $\proj$, there is a natural transformation $\proj_*\to \Delta_\catA^*$. It is straightforward to verify that the functor $\proj_{*}$ is lax monoidal  (similar to Lemma~\ref{lem:lax}). Moreover, there is an algebra homomorphism
\[\bbS\to \proj_{*}\left(
\bbS\boxtimes (\calO\rtimes W)\right). \]
Considered as coherent sheaves on $\catA\times \catA^\vee$ via pushforward in \S~\ref{subsec:loc}, locally in terms of sections this morphism is given by 
\begin{equation}\label{eqn:psi}
 f\delta_w\delta^\dyn_v\mapsto  f\delta_w\delta^\dyn_v\otimes \delta_w, \quad w,v\in W.
\end{equation}

Let $j^2$ be the open embedding of the regular locus of $\catA\times_{\catA/W} \catA\times \catA^\vee\times_{\catA^\vee/W^\dyn }\catA^{\vee}$. Similar as in \S~\ref{subsec:loc}, 
we have the induced algebra homomorphism 
of quasi-coherent sheaves
\begin{equation}
\label{eq:psi}\Psi: j^2_*(j^2)^*
\bbS\to j^2_*(j^2)^*\proj_*\big(\bbS\boxtimes (\calO\rtimes W)\big).\end{equation}
Composing with the projection $j^2_*(j^2)^*\proj_*\big(\bbS\boxtimes(\calO\rtimes W)\big)\to j^2_*(j^2)^*\Delta_\catA^*\big(\bbS\boxtimes (\calO\rtimes W)\big)$, below we find the image  of the subsheaf $ \calH^\dyn \subseteq j^2_*(j^2)^*\bbS$. 
\begin{lemma}\label{lem:coprod}
The morphism $\Psi$ induces  a morphism of coherent sheaves on $\catA\times\catA^\vee$ \[\Psi:\calH^{\dyn}\to (\calH^{\dyn}+\bbS)\otimes p^*\calH.\]
\end{lemma}
Here the sum $\calH^{\dyn}+\bbS$ means the subsheaf in $j_*j^*\bbS$ generated by $ \calH^{\dyn}$ and $\bbS$.
\begin{proof}
By Lemma \ref{lem:H_basic}, it suffices to verify locally on sections of the form $T:=a+b\de_\al$.
Here $b$ is a rational section of $\bbS_{s_\al,e}$ and has an order-1 pole at $D_\al$, and an  order-1  zero at $D_{\hbar, \al}$.
Then, the map \eqref{eqn:psi} can be rewritten as 
\[\Psi(T)=a\otimes 1+ b\de_\al\otimes \de_\al= T\otimes 1+ b \de_\al\otimes (\de_\al-\id). \]
Locally, we write $b$ as $d\cdot \frac{b}{d}$ where $ \frac{b}{d}$ is a regular section of $\bbS_{s_\al,s_\al}$, and $d$ is a rational function on $\catA$ with  an order-1 pole at $D_\al$, and an  order-1  zero at $D_{\hbar, \al}$.
 Since the tensor is over $\calO_{\catA\times\catA^\vee}$, we have $ b\de_\al\otimes (\de_\al-\id)= \frac{b}{d}  \de_\al\otimes d(\de_\al-\id)$. 
Now $d(\de_\al-\id)$ satisfies the GKV residue condition and hence lands in $p^*\calH$, and $\frac{b}{d} \de_\al$ lands in $\bbS$.
\end{proof}

\subsubsection{Step 2}
We write $\Delta_\catA:\catA\times \catA^\vee\to (\catA\times \catA^\vee)\times_{\catA/W}\catA$, $(a,b)\mapsto (a,b, a)$. 
Denote the projection $\proj: (\catA\times \catA^\vee)\times_{\catA/W}\catA\to  \catA\times \catA^\vee$. Since $\proj\circ \Delta_{\catA}=\id$, there is a natural transformation $\proj_*\to \Delta_\catA^{*}$,  
Notice that for any  $\calF$ in $\Coh(\catA)$ and $\calG$
 in $\Coh(\catA\times \catA^\vee)$, clearly 
 \begin{equation}
\label{eq:Delta}\Delta_\catA^{*}(\calG\boxtimes \calF)=\calG\otimes p^*\calF,\end{equation}
which is a quotient of $\proj_{*}(\calG\boxtimes \calF)$.

Now let $\calF$ in $\Coh(\catA)$ be an module of $\calH$. 
Let $T:=a+b\de_\al$ be a local section of $\calH^\dyn$. Let us consider 
 \[T'=T\boxtimes 1+ \frac{b}{d}  \de_\al\boxtimes d(\de_\al-\id),\]
where $b,d$ are described as in the proof of Lemma~\ref{lem:coprod}. 
Then on any open subset of $(\catA\times\catA^\vee)\times_{\catA/W}\catA$ on which  $T'$ is well-defined as a regular section  of $(\calH^\dyn+\bbS)\boxtimes \calH$, we have a map
\[T':\pi^2_*\proj_*(\bbL\boxtimes \calF)\to \pi^2_*\proj_*(\bbL\boxtimes \calF).\]

\begin{lemma}
Let $T$ be a regular section  of $\calH^\dyn $ on an open subset of $\catA\times\catA^\vee$.  Then, on this open set, $T'$ induces a map on the quotients
\[\Psi(T):\pi^2_*(\Delta_\catA^{*}(\bbL\boxtimes \calF))\to \pi^2_*(\Delta_\catA^{*}(\bbL\boxtimes \calF)).\]
\end{lemma}
\begin{proof}
It suffices to verify on a local section of $\pi^2_*(\Delta_\catA^{*}(\bbL\boxtimes \calF))$ of the form  $fm\otimes n=m\otimes fn$ with $f$ a function on $\catA$, $m$ a rational section of $\bbL$, and $n$ a rational section of $\calF$. We write $T$ as  $\sum_{v,w}a\de_w\de_v^\dyn$, so that $\Psi(T)=\sum_{v,w} a\de_w\de_v^\dyn\otimes \de_w$. We have 
\[
\Psi(T)(fm\otimes n)=\sum_{v,w}a\cdot {}^{wv^\dyn}f \cdot {}^{wv^\dyn}m\otimes{}^w n, \quad \Psi(T)(m\otimes fn)=\sum_{v,w}a\cdot {}^{wv^\dyn}m\otimes {}^{w}f\cdot {}^{w}n. 
\]
Since $f$ is a function on $\catA$, ${}^{v^\dyn}f=f$, so we have $\Psi(T)(fm\otimes n)=\Psi(T)(m\otimes fn)$. \end{proof}

Using \eqref{eq:Delta}, we see that a rational section $T$ of $\calH^\dyn$ induces a map 
\[
\Psi(T):\pi_*^2(\bbL\otimes p^*\calF)\to \pi_*^2(\bbL\otimes p^*\calF). 
\]This concludes the proof of Theorem~\ref{thm:pullback}. 
\begin{example}
We can write the action of $\calH^{\dyn}$ on $\bbL\otimes p^*\calO_{\catA}$, on the level of rational sections as follows. 
Let $T:=a\de_\al^\dyn+b\de_\al\de_{\al}^\dyn$. Let $f$ be a rational section of  $p^*\calO_{\catA}$ (so that $s_\al^\dyn$ acts trivially), and $g$ be a rational section of $
\pi^2_*\bbL$,  then the map $\Psi$ induces \eqref{eqn:psi}:
\[\Psi(T)( g\otimes f ):=T(g)\otimes f+b\cdot {}^{s_\alpha^\dyn s_\alpha}g\otimes \big({}^{s_\alpha}f-f\big).\]
 Notice that $\bbL$ is a module over $p^*\calO_{\catA}$, hence we have the action map $\bbL\otimes p^*\calO_{\catA}\to\bbL$. So if $f$ is a rational section of $p^*\calO_{\catA}$, we have 
\begin{align*}
\Psi(T)(g\otimes f)&=(T\otimes 1+b\de_\al \de_\al^\dyn\otimes (\de_\al-\id))(g\otimes f)\\
&= (a\cdot {}^{s_\al ^\dyn}g+b\cdot {}^{s_\al s_\al ^\dyn}g)\cdot f+b\cdot {}^{s_\al s_\al ^\dyn}g\cdot ({}^{s_\al }f-f)\\
&=a\cdot f\cdot {}^{s_\al ^\dyn}g+b\cdot {}^{s_\al}f\cdot{}^{s_\al s_\al ^\dyn}g\\
&=T(fg).
\end{align*}
In other words, the multiplication map  $\bbL\otimes p^*\calO_{\catA}\to \bbL$ is an $\calH^\dyn$-module homomorphism with the action on the domain defined via $\Psi$.
\end{example}

\subsection{From dynamical Hecke algebra to Langlands dual GKV algebra}
We consider the functor going from the module category of the dynamical Hecke algebra to that of the Langlands dual GKV algebra. 
\begin{lemma}
\label{thm:Weyl}
In the category $\Coh^{W^\dyn }(\catA\times_{\catA/W}\catA\times(\catA^\vee)^{\reg})$, there is an algebra homomorphism \[\calO_{(\catA^\vee)^{\reg}}\rtimes W\to q_*\calH^{\dyn}|_{\catA\times_{\catA/W}\catA\times(\catA^\vee)^{\reg}}\] sending $\de_\alpha^\dyn $ to $T_\alpha$.
\end{lemma}
Here $T_\alpha$ is recalled in \S~\ref{subsec:DL}. For simplicity, denote $T_\al=\fp\de_\al^\dyn+\fq\de_\al\de_\al^\dyn$.
\begin{proof}
We claim that the natural inclusion $\pi_*\calO_{(\catA^\vee)^{\reg}}\to \pi_*q_*\calH^{\dyn}|_{\catA\times_{\catA/W}\catA\times(\catA^\vee)^{\reg}}$, together with the map of global sections $W\ni\de_\alpha^\dyn \mapsto T_\alpha$ defines a morphism of sheaves of algebras 
\[\calO_{(\catA^\vee)^{\reg}}\rtimes W^\dyn \to q_*\calH^{\dyn}|_{\catA\times_{\catA/W}\catA\times(\catA^\vee)^{\reg}}.\]
Indeed, $\pi_*\calO_{(\catA^\vee)^{\reg}}\rtimes W^\dyn $ is generated by $\pi_*\calO_{(\catA^\vee)^{\reg}}$ and $W^\dyn $, the latter of which in turn is generated by the $\de_\alpha^\dyn $'s. To show the above claim, it suffices to verify the commutation-relations among the  the $\de_\alpha$'s, as well as the commutation-relations between each $\de_\alpha^\dyn $ and $\pi_*\calO_{(\catA^\vee)^{\reg}}$. More precisely, the latter means that $\de_\al^\dyn f={}^{s_\al^\dyn}f \de_\al^\dyn$ for a rational section $f$ of $\pi_*\calO_{(\catA^\vee)^\reg}$. 

Thanks to \cite{RW20} and \cite[Proposition~4.11]{ZZ15}, $T_\alpha$'s satisfy relations of $W$. 
Let $f$ be the pullback of a function on $U\subseteq \calA^\vee$ (hence ${}^{s_\al }f=f$) and $g$ a section of $\bbL$ on any open subset of $\catA\times U$. We have
\begin{align}
T_\alpha(fg)&=\big(\fp\de_\alpha^\dyn +\fq \de_\alpha \de^\dyn _\alpha\big)(fg)=\fp \cdot {}^{s_\al^\dyn}f\cdot {}^{s_\al^\dyn}g
+\fq\cdot {}^{ s_\al^\dyn}f \cdot {}^{s_\al s_\al^\dyn}g={}^{s_\al^\dyn}f\cdot T_\al(g).
\end{align}
By Theorem \ref{thm:poly}, we get that 
\[T_\alpha f={}^{s^\dyn _\alpha}fT_\alpha.\]
This is precisely the commutation relation between $\pi_*\calO_{(\calA^\vee)^\reg}$ and $W^\dyn $. 
\end{proof}

Recall that we denote by $\calH^L$  the GKV elliptic Hecke algebra associated to the Langlands dual root system.

\begin{lemma}\label{lem:global}
We have an algebra homomorphism $\Gamma: \calH^L\to q_*\calH^{\dyn}$ in $\Coh(\catA^\vee\times_{\catA^\vee/W^\dyn }\catA^\vee)$. 
\end{lemma}
\begin{proof}
By Lemma~\ref{thm:Weyl}, we have $j^\vee_*(j^\vee)^*\pi_*\calO_{\catA^\vee}\rtimes W^\dyn \to j^\vee_*(j^\vee)^*q_*\calH^{\dyn}$. 
It suffices  to show that the image of the subalgebra $\calH^L\subseteq j^\vee_*(j^\vee)^*\pi_*\calO_{\catA^\vee}\rtimes W^\dyn $ lands in $q_*\calH^{\dyn}\subseteq j^\vee_*(j^\vee)^*q_*\calH^{\dyn}$. 
The algebra $\calH^L$ is generated locally by the filtered pieces labeled by simple roots. It suffices to show the image of such generators, say, $\sigma=a+b\de_{\al}^\dyn$, are regular sections of $q_*\calH^{\dyn}$. By definition, $\sigma$ is mapped to 
\begin{equation}\label{eq:sigma}
\Gamma(\sigma)=a+bT_\al=a+\frac{b\theta(\hbar)\theta(z_\al-\la_{\al^\vee})}{\theta(\hbar-\la_{\al^\vee})\theta(z_\al)}\de_\al^\dyn+\frac{b\theta(\la_{\al^\vee})\theta(\hbar-z_\al)}{\theta(z_\al)\theta(\hbar-\la_{\al^\vee})}\de_\al\de_\al^\dyn.
\end{equation}

Denote the latter two fractions by $c_1,c_2$, respectively. We check that the GKV residue conditions for $\calH^\dyn$ are satisfied. Since $\sigma$ satisfies the residue conditions for the Langlands dual system, so  $a$ and $b$ have possible simple poles at $D_{\al^\vee}$,  $b$ has an order one zero at the divisor $D_{\hbar, \al^\vee}$ and $\Res_{\al^\vee}(a)+\Res_{\al^\vee}(b)=0$.

Due to cancellation with the zero and pole of $b$, the fraction $c_2$ has only an order one pole at $D_\al$, and an order one zero at $D_{\hbar, \al}$. So Condition iii) is satisfied. 
The possible poles of $a,c_1$ are at the divisors $D_\al$ and $D_{\al^\vee}$.  So Condition i) is satisfied.   For Condition ii), clearly $\Res_\al (a)=0$, and it is easy to see that $\Res_{\al}(a)+\Res_\al(c_1)=0$.  Lastly, we check Condition ii'). We have
\begin{align*}
[\theta(\la_{\al^\vee})\left(a+\frac{b\theta(\hbar)\theta(z_\al-\la_{\al^\vee})}{\theta(\hbar-\la_{\al^\vee})\theta(z_\al)}\right)]|_{\la_{\al^\vee}=0}&=\Res_{\al^\vee}(a)+\Res_{\al^\vee}(b)\cdot (\frac{\theta(\hbar)\theta(z_\al-\la_{\al^\vee})}{\theta(\hbar-\la_{\al^\vee})\theta(z_\al)})|_{\la_{\al^\vee}=0}\\&=\Res_{\al^\vee}(a)+\Res_{\al^\vee}(b)=0.
\end{align*}
 The proof is finished.
\end{proof}
\begin{corollary}\label{cor:proj}
There is a functor $\Mod\calH^\dyn \to \Mod\calH^{L}$ making the following diagram commutative
\[\xymatrix{
\Mod\calH^{\dyn}\ar[r]\ar[d]& \Mod\calH^L\ar[d]\\
 \Coh\catA\times\catA^\vee\ar[r]_-{q_*}&  \Coh\catA^\vee.
}\]
\end{corollary}
\begin{proof}
By a similar argument as that of Lemma~\ref{lem:lax}, the functor $q_*$ is lax monoidal. Therefore it sends an $\calH^{\dyn}$-module to a $q_*\calH^{\dyn}$-module and hence an $\calH^L$-module.
\end{proof}

\begin{remark}\label{rem:higher_derived}
Similar argument as in Corollary~\ref{cor:proj} shows that for any $r\geq 0$,  the functor $R^r q_*$ sends  an $\calH^{\dyn}$-module to a $q_*\calH^{\dyn}$-module and hence an $\calH^L$-module. Therefore, each $R^r q_*$ extends to a functor $\Mod\calH^\dyn \to \Mod\calH^{L}$. 
\end{remark}

The following is not used in the rest of the paper, but is  included for completeness.

\begin{prop}
The map in Lemma~\ref{thm:Weyl} is an isomorphism.  Consequently, the  group of global sections  of $\calH^{\dyn}|_{\catA\times_{\catA/W}\times(\catA^\vee)^\reg}$ is free of finite rank, with  basis $T_w$, $w\in W$.
\end{prop}
\begin{proof}
To show the statement about isomorphism, it suffices to show surjectivity, which is can be done at the stalk  of each point in $(\catA^\vee)^{\reg}/W^\dyn $. 
\end{proof}

\begin{remark}
It may be interesting to find $R^r q_*\calH^{\dyn}$ for all $r\geq0$. We refrain from pursuing this question in the present paper. 
\end{remark}

\section{Conclusions}\label{sec:conclutions}
\subsection{Fourier-Mukai transform of representations}
The composition $FM:=q_*\circ (\bbL\otimes -)\circ p^*: \Coh\catA\to\Coh\catA^\vee$ is the Fourier-Mukai transform. 
It is well-known that it induces an equivalence between the abelian category of finitely supported coherent sheaves on $\catA$ and that of flat vector bundles on $\catA^\vee$.

\begin{corollary}\label{cor:fm}
There is a functor 
$\Mod\calH\to \Mod\calH^L$ making the following diagram commutative
\[\xymatrix{
\Mod\calH\ar[d]\ar[r]&\Mod\calH^L\ar[d]\\
\Coh\catA\ar[r]_{FM}&\Coh\catA^\vee.
}\]
\end{corollary}
\begin{proof}
The composition of the functors from Theorem~\ref{thm:pullback} and Corollary~\ref{cor:proj} gives the asserted functor. 
\end{proof}

Remark~\ref{rem:higher_derived} shows that $R^r q_*\circ (\bbL\otimes -)\circ p^*$ for each $r\geq 0$ defines a functor $\Mod\calH\to \Mod\calH^L$.

\subsection{The inverse Fourier-Mukai transform}
Denote $\bbL^*$ the dual of the line bundle $\bbL$ (which is also the inverse of $\bbL$). 
 Denote the corresponding dyanmical Hecke algebra for the Langlands dual system satisfying the residue conditions in Definition \ref{def:res} by $\calH_{-\hbar}^{\dyn, L}$, then by \cite[Propositoin 5.1 and Theorem 6.3]{LZZ23}, together with the residue conditions, we see that $\calH_{-\hbar}^{\dyn, L}$ contains the operator $T^L_\al:=T^{z,\la,\dyn}_{\al}$ (defined in \cite[(26)]{LZZ23}) as a rational section.  

Then the map corresponding to $\Psi$  in Lemma \ref{lem:coprod} is 
\[\Psi^L:\calH_{-\hbar}^{\dyn, L}\to  (\calH_{-\hbar}^{\dyn, L}+\bbS)\otimes q^*\calH^L, \quad \sum_{v,w}f\de_w\de_v^\dyn\mapsto \sum_{v,w}f\de_w\de_v^\dyn\otimes\de_v^\dyn.\]
And the map corresponding to $\Gamma$ in Lemma \ref{lem:global} is 
\[
\Gamma^L:\calH\to p_*\calH^{\dyn, L}_{-\hbar}, ~\sum_{v}a\de_v\mapsto \sum_{v}aT^L_v. 
\]

We  consider the inverse Fourier-Mukai transform of coherent sheaves  \[\widehat{FM}:=[-]_{\catA}^*\circ R^n p_*\circ (\bbL\otimes -)\circ q^*,\] where 
$n$ is the rank of the root system and $[-]_{\catA}:\catA\to \catA$ the inverse map of the abelian variety.
\begin{corollary} 
There is a functor 
$\Mod\calH^L\to \Mod\calH$ making the following diagram commutative
\[\xymatrix{
\Mod\calH^L\ar[d]\ar[r]&\Mod\calH\ar[d]\\
\Coh\catA^\vee\ar[r]_{\widehat{FM}}&\Coh\catA.
}\]
\end{corollary}
\begin{proof} By definition the functor $(\bbL\otimes -)\circ q^*$ from $\Mod\calH^L$ lands in the category $\Mod \calH^{\dyn,L}_{-\hbar}$. Similarly $p_*$ goes from $\Mod \calH^{\dyn,L}_{-\hbar}$ to $\Mod \calH_{-\hbar}$. Finally, we have an isomorphism \[[-]_{\catA}^*\calH_{-\hbar}\cong \calH.\]
Therefore, $\widehat{FM}$ extends to $\Mod\calH^L\to \Mod\calH$. 
\end{proof}

\subsection{Comparison with the classical FM transform}
Recall that the usual Fourier-Mukai transform \cite[\S~11.2]{P} is defined a priori on the derived categories, and using the Poincar\'e line bundle $\calP$ as the kernel. 
In the present paper,  we use  $\bbL$ as the kernel. Recall that $\bbL=\calP\otimes(\calO(\hbar\rho)\boxtimes\calO(-\hbar\rho^\vee))$.
Therefore, a comparison is in order.
\begin{remark}
For simplicity, we first consider functors on derived categories. All functors in this remark are derived unless otherwise specified. 
The usual Fourier-Mukai transform $FM^0$ is defined as $q_*\circ(-\otimes\calP)\circ p^*:D^b\Coh(\catA)\to D^b\Coh(\catA^\vee)$. 
  Therefore, by projection formula we have
\[FM=(-\otimes\calO(-\hbar\rho^\vee))\circ FM^0\circ(-\otimes\calO(\hbar\rho)).\]
Similarly, the usual inverse Fourier-Mukai transform is $\widehat{FM}^0=[-]_{\catA}^*\circ p_*\circ [-n]\circ(\calP\otimes -)\circ q^*$, where $[-n]$ is homological shift. Again, the projection formula gives
\[\widehat{FM}=(-\otimes \calO(-\hbar\rho))\circ\widehat{FM}^0\circ(-\otimes\calO(-\hbar\rho^\vee)).\]
 The equality $FM^0\circ \widehat{FM}^0=\id$  \cite[Theorem~11.6]{P} then implies that
\begin{align*}
FM\circ \whFM(\calF)&=\calO({-\hbar \rho^\vee})\otimes FM^0\big(\calO({\hbar\rho})\otimes [\calO({-\hbar\rho})\otimes \whFM^0(\calF\otimes \calO({-\hbar\rho^\vee}))]\big)=\calF({-2\hbar\rho^\vee}). 
\end{align*}
Here we denote $\calF(-2\hbar\rho^\vee)=\calF\otimes\calO(-2\hbar\rho^\vee)$ for simplicity. 

 We explicitly describe the natural isomorphism \[-\otimes\calO(-2\hbar\rho^\vee)\to FM\circ \widehat{FM}:D^b\Coh\catA^\vee\to D^b\Coh\catA^\vee \] as follows. For any object $\calF$ in $D^b\Coh\catA^\vee$ (resp.  $\calG$  in $D^b\Coh\catA$) we write $\bbL\otimes_{\catA^\vee}\calF$ (resp. $\bbL\otimes_\catA \calG$) for the object $\bbL\otimes q^*\calF$ (resp. $\bbL\otimes p^*\calG$) on $\catA\times\catA^\vee$. Then, simplifying using base-change and projection formula, we get \[FM\circ \widehat{FM}(\calF)=q_*\left((p^*p_*\calO)\otimes(\bbL[-r]\otimes[-]^*_\catA\bbL\otimes_{\catA^\vee}\calF)\right)\]
By calculation (see also \cite[Lemma~4.4(5)]{LZZ23}), $[-]^*_\catA\bbL\cong\bbL^*\otimes_{\catA^\vee}\calO(-2\hbar\rho^\vee)$. 
The  identity section $\calO\to \bbL^*\otimes\bbL$
 gives \begin{equation}\label{eqn:id_sec}q^*\calF(-2\hbar\rho^\vee)\to \bbL\otimes\bbL^*\otimes_{\catA^\vee}\calF(-2\hbar\rho^\vee).
\end{equation}
Tensoring with the
Serre duality  $\calO\to p^*p_*\calO[-r]$ gives \[q^*\calF(-2\hbar\rho^\vee)\to (p^*p_*\calO[-r])\otimes(\bbL\otimes\bbL^*\otimes_{\catA^\vee}\calF(-2\hbar\rho^\vee)).\]
Applying $q_*$ then yields the map $\calF(-2\hbar\rho^\vee)\to FM\circ \widehat{FM}(\calF)$.

\end{remark}

As remarked above, if $\calF$ is a complex of $\calH$-modules,  taking the $r$th cohomology of the complex $FM(\calF)$ for any $r$ gives a $\calH^L$-module. The functor on abelian categories from Corollary~\ref{cor:fm} is obtained by taking the zero-th cohomology of the complex $FM(\calF)$. 
Nevertheless, if $\calF$ is a flat vector bundle as a coherent sheaf on $\catA$, then the complex $FM(\calF)$ is concentrated in homological degree zero. Furthermore, its zero-th cohomology has finite support as a coherent sheaf on $\catA^\vee$ \cite[Proposition~11.8]{P}. Therefore, we obtain the following lemma.
\begin{lemma}
\begin{enumerate}
\item 
The functor $FM$ from  Corollary~\ref{cor:fm} restricts to an exact functor  between abelian categories $\Mod_{fl}\calH\to \Mod_{fin}\calH^L$.
\item Similarly, the functor $\widehat{FM}$ restricts to the subcategories $\Mod_{fin}\calH^L\to \Mod_{fl}\calH$. 
\end{enumerate}
\end{lemma}

We also have the following. 
\begin{lemma}\label{lem:adj}
The canonical isomorphism $-\otimes\calO(-2\hbar\rho^\vee)\to FM\circ \widehat{FM}:\Coh_{fin}\catA^\vee\to \Coh_{fin}\catA^\vee$ extends to $\Mod_{fin}\calH^L\to \Mod_{fin}\calH^L$.
\end{lemma}

\begin{proof}
We write $T_\alpha$ from \S~\ref{subsec:DL} as $\fp\de_\al^\dyn+\fq\de_\al\de_\al^\dyn$ for short; and  the counterpart for the Langlands dual system as $T_\alpha^L=\fp^L\de_\al+\fq^L\de_\al\de_\al^\dyn$. One can verify easily (see also \cite[Theorem 7.13]{LZZ23}) that 
\begin{equation}\label{eq:inverse}\fp+\fq\cdot {}^{s_\al s_\al^\dyn}\fp^L=0, \quad \fq\cdot{}^{s_\al s_\al^\dyn}\fq^L=1.\end{equation}
We need to show that for a local section of $\calH^L$ of the form $T=a(\lambda)+b(\lambda)\delta^\dyn_\al=a+b\de_\al^\dyn$, its action on a $\calH^L$-module $\calF$ over $\catA^\vee$ is compatible with the action of $(\id_{\calH^\dyn}\otimes_{\catA}\Psi^L)\circ(\id_{\calH^\dyn}\otimes_\catA\Gamma^L)\circ\Psi\circ\Gamma(T)$ via the map \eqref{eqn:id_sec}
\[q^*\calF(-2\hbar\rho^\vee)\to \bbL\otimes\bbL^*\otimes_{\catA^\vee}\calF(-2\hbar\rho^\vee).\]

Unwrapping the definitions of $\Psi,\Psi^L, \Gamma, \Gamma^L$ from \eqref{eq:psi} and \eqref{eq:sigma}, we have
\begin{align*}
&(\id_{\calH^\dyn}\otimes_{\catA}\Psi^L)\circ(\id_{\calH^\dyn}\otimes_\catA\Gamma^L)\circ\Psi\circ\Gamma(a+b\de^\dyn_\al)\\
&=a\otimes_\catA 1\otimes_{\catA^\vee}1+b\fp\de_\al^\dyn\otimes_\catA 1\otimes_{\catA^\vee}1+b\fq\de_\al\de_\al^\dyn\otimes_\catA \fp^L\de_\al \otimes_{\catA^\vee} 1+ b\fq\de_\al\de_\al^\dyn\otimes_\catA \fq^L\de_\al\de_\al^\dyn\otimes _{\catA^\vee}\de_\al^\dyn\\
&\overset{\sharp_1}=a\otimes_\catA 1\otimes_{\catA^\vee}1+b\fp\de_\al^\dyn \otimes_\catA 1\otimes_{\catA^\vee}1+b\fq\cdot{}^{s_\al s_\al^\dyn}\fp^L\de_\al^\dyn\otimes_{\catA}1 \otimes_{\catA^\vee}1+\fq\cdot{}^{s_\al s_\al^\dyn}\fq^L\otimes_{\catA}1\otimes_{\catA^\vee}b\de_\al^\dyn\\
&\overset{\sharp_2}=1\otimes_\catA 1\otimes_{\catA^\vee}a+1\otimes_{\catA}1\otimes_{\catA^\vee}b\de_\al^\dyn\\
&=1\otimes_{\catA}1\otimes_{\catA^\vee}T.
\end{align*}
The identity $\sharp_2$ follows from \eqref{eq:inverse}. For the  identity $\sharp_1$,
Note that from the isomorphism $\bbL\otimes \bbL^*=\bbL\otimes \sHom(\bbL,\calO)\cong \calO, f\otimes g\mapsto g(f)$, we have 
\[(\de_\al\otimes \de_\al)(f\otimes g)=s_\al(g)(s_\al(f))=g(s_\al s_\al (f))=(1\otimes 1)(f\otimes g), \]
and similarly we have $\de_\al^\dyn\otimes \de_\al^\dyn=1\otimes 1$ as operators  on $\bbL\otimes \bbL^*$. 
Therefore, \[
b\fq\de_\al\de_\al^\dyn\otimes_{\catA}\fp^L\de_\al\otimes _{\catA^\vee}1=b\fq\cdot{}^{s_\al s_\al^\dyn}\fp^L\de_\al^\dyn\otimes_\catA 1\otimes_{\catA^\vee}1, ~b\fq\de_\al\de_\al^\dyn \otimes_\catA \fq^L\de_\al\de_\al^\dyn\otimes_{\catA^\vee} \de_\al^\dyn=\fq \cdot{}^{s_\al s_\al^\dyn} \fq^L\otimes_\catA 1\otimes_{\catA^\vee}b\de_\al^\dyn.
\]
They imply the identity $\sharp_1$. The proof is finished.
\end{proof}
We refer an interested reader to \cite[Theorem~7.13]{LZZ23} for a generalization of \eqref{eq:inverse}, and its implications on elliptic classes. 

\subsection{Equivalence of categories}
Combing the above, we have. 
\begin{corollary}
Restricting to subcategories $\Mod_{fl}\calH\to \Mod_{fin}\calH^L$, we get an equivalence of abelian categories. 
\end{corollary}

\begin{proof}
The forgetful functor 
 $\Mod_{fin}\calH^L\to \Coh_{fin}\catA^\vee$ has a left adjoint given by the induction functor
\[\calF\mapsto 
\calH^L\star\calF.\]
 Since $\calH$ is flat as a coherent sheaf on $\catA$ (see Lemma~\ref{lem:flat}), 
the forgetful functor $\Mod_{fl}\calH\to \Coh_{fl}\catA$ has a similar left adjoint. 
The functors $FM$ and $\widehat{FM}$ induce equivalences between  $\Coh_{fl}\catA$ and $\Coh_{fin}\catA^\vee$ \cite[Proposition~11.8]{P}, hence keeping in mind Lemma~\ref{lem:adj}, they  induce equivalences between $\Mod_{fl}\calH$ and $\Mod_{fin}\calH^L$ by Beck's monadicity theorem \cite[Theorem~VI.7.1]{Mac}.
\end{proof}

Let $G$ be the connected, complex reductive group associated to the root system. An $\calO(\hbar)$-values $G$-Higgs bundle is a pair $(P,x)$ where $P$ is a semi-simple semi-stable degree-0 principle $G$-bundle, and $x$ is a section of $\ad(P)\otimes\calO(\hbar)$. We say $(P,x)$ is nilpotent if the image of $x$ lands in the nilpotent cone at each fiber. Let $\Aut(P,x)$ be the group of automorphisms of $(P,x)$, and $C(P,x):=\Aut(P,x)/\Aut^0(P,x)$ the component group. Let $\bfB_{P,x}$ be the variety parametrizing all the Borel structures on $(P,x)$. Clearly, $C(P,x)$ acts on $H^*(\bfB_{P,x};\bbC)$. 
Let $Higgs^{G}_{\hbar}(E)$ be the set of triples $(P,x,\chi)$, where $(P,x)$ is nilpotent and $\chi$  is an
 irreducible $C(P,x)$-module with non-zero multiplicity in $H^*(\bfB_{P,x};\bbC)$.
\begin{corollary}
If the group $G$  is of adjoint type, then  irreducible objects in $\Mod_{fl}\calH$ are in one-to-one correspondence with $Higgs_{\hbar}^{G^L}(E)$.
\end{corollary}
\begin{proof}
This is a direct consequence of Corollary~\ref{cor:fm} and \cite[Theorem~B]{ZZ15}.
\end{proof}

As explained in \cite[\S~7.1]{ZZ15}, representations of $\Mod\calH$ and those of double affine Hecke algebra of Cherednik are expected to be related. 
In the same vein, the Fourier-Mukai transform in the present paper is expected to be compatible with the Fourier transform of the double affine Hecke algebra, known to be related to $S$-duality and classical geometric Langlands duality \cite{Kap}. We leave the relation with  double affine Hecke algebra  to future investigations. 

\appendix
\section{Reflexive sheaves} \label{sec:App}
In this Appendix we collect some facts from commutative algebra and algebraic geometry that are used in the proof of Lemma \ref{lem:H_basic}. 

Let $R$ be a Noetherian local ring, and $M$  a finitely generated $R$-module. 
Recall that $M$ is called reflexive if $(M^*)^*\cong M$.  Let $\omega$ be the dualizing complex. Then, $M$ is maximal Cohen-Macaulay (MCM) if and only if $\Ext_R^q(M,\omega)$ vanishes for all but a single $q\in \bbZ$  \cite[\href{https://stacks.math.columbia.edu/tag/0B5A}{Tag 0B5A}]{Stack}.

The followings are well-known.
\begin{lemma}\label{lem:ref}
\begin{enumerate}
\item $M$ is reflexive if and only if it is torsion free and $M\cong \cap _{\frakp}M_\frakp$ where the intersection is taken among all prime ideals $\frakp$ with height 1(\cite[7.4.2]{Bou98}, see also \cite[\href{https://stacks.math.columbia.edu/tag/0AUY}{Tag 0AUY}]{Stack}); 
\item any MCM module is reflexive (see, e.g., \cite[Satz 6.1]{HK}).  
\end{enumerate}
\end{lemma}

We say a coherent sheaf  is reflexive (resp. MCM, resp. torsion-free) if its stalk at each point is reflexive (resp. MCM, resp. torsion-free). 
For any Noetherian scheme $Y$, let 
$\omega_Y$ be the normalized dualizing complex \cite[\href{https://stacks.math.columbia.edu/tag/0AX1}{Tag 0AX1}]{Stack}.

\begin{lemma}\label{lem:MCM} Let $X$ be a smooth quasi-projective variety on which the finite group $W$ acts.  Let $\pi:X\to X/W=Y$ be the canonical projection. Let $\calF$ be  a vector bundle over $X$.  Then 
\begin{enumerate}
\item $R^n\pi_*\calF=0$ for $n>0$;
\item $\sExt^n_Y(\pi_*\calF, \omega_Y)=0$  for $n>0$;
\item 
In particular,  $\pi_*\calF$ is MCM. 
 \end{enumerate}
\end{lemma}
\begin{proof}
Since $W$ is finite,  $\pi$ is a finite morphism. This implies that $\pi_*$ is an exact functor. Therefore, $R^n\pi_*\calF=0$ for $n>0$. 

We now have $R\pi_*\calF\cong \pi_*\calF$. 
From the Grothendieck duality, we have  
\[R\sHom_Y(R\pi_*\calF,\omega_Y)\cong R\pi_*R\sHom_X(\calF,\omega_X).\]
We calculate the right hand side using the Grothendieck spectral sequence. Note that 
\[
R^p\pi_*\sExt^q_X(\calF,\omega_X)=0\] for $p>0$. We obtain \[
\sExt_Y^q(\pi_*\calF,\omega_Y)\cong \pi_*\sExt_X^q(\calF,\omega_X)\cong \begin{cases} 0, & \text{ if }q>0;\\
\pi_*(\calF^\vee\otimes \omega_X), & \text{ if }q=0.
\end{cases}
\]
By   \cite[\href{https://stacks.math.columbia.edu/tag/0B5A}{Tag 0B5A}]{Stack}, $\pi_*\calF$ is MCM.
\end{proof}

Finally, we remark that in the same setup as in Lemma~\ref{lem:MCM}, $\pi_*\calF$ is  locally free if and only if $Y$ is smooth. Indeed, since $\pi$ is a finite morphism, $\pi_*\calF$ being locally free is equivalent to $\calF$ being flat over $Y$, which in turn is equivalence to $\pi$ being a flat morphism since $\calF$ is locally free on $X$. Since $X$ is projective and $Y$ is connected, if $\pi$ is flat, it is faithfully flat. Since $X$ is flat over $\Spec \bbC$, so is $Y$ \cite[\href{https://stacks.math.columbia.edu/tag/01U2}{Tag 01U2}]{Stack}. Conversely, if $Y$ is smooth, then by miracle flatness the map $\pi$ is flat.

\newcommand{\arxiv}[1]
{\texttt{\href{http://arxiv.org/abs/#1}{arXiv:#1}}}
\newcommand{\doi}[1]
{\texttt{\href{http://dx.doi.org/#1}{doi:#1}}}
\renewcommand{\MR}[1]
{\href{http://www.ams.org/mathscinet-getitem?mr=#1}{MR#1}}


\begin{thebibliography}{00}
\bibitem[AM10]{AM10} M. Aguiar, S. Mahajan, {\em Monoidal Functors, Species and Hopf Algebras}, With forewords by Kenneth Brown and Stephen Chase and Andr\'e Joyal. CRM Monograph Series, 29. American Mathematical Society, Providence, RI, 2010. lii+784 pp. ISBN: 978-0-8218-4776-3.


\bibitem[Bou98]{Bou98} N Bourbaki, {\em Commutative Algebra: Chapters 1-7}, Springer Berlin, 1998.

\bibitem[Ha77]{Hart} R. Hartshorne, {\em Algebraic Geometry}, GTM, volume 52, Springer, 1977.

\bibitem[HK71]{HK}
 J. Herzog, E. Kunz, {\em Der kanonische Modul eines Cohen-Macaulay-Rings}, LNM, volume 238,
 (1971)


\bibitem[GKV97]{GKV97}
V.~Ginzburg, M.~ Kapranov, and E.~Vasserot,
{\em Residue construction of Hecke algebras}, Adv. Math. {\bf 128} (1997), no. 1, 1--19. \MR{1451416}


\bibitem[K06]{Kap} A. Kapustin, {\em Holomorphic reduction of N=2 gauge theories, Wilson-'t Hooft operators, and S-duality}, 2006. \arxiv{0612119}


\bibitem[LZZ23]{LZZ23} C. Lenart, G. Zhao and C. Zhong, {\em Elliptic classes via the periodic Hecke module and its Langlands dual}, preprint, 
\arxiv{2309.09140}.

\bibitem[M71]{Mac} S. MacLane, {\em Categories for the Working Mathematician}, Graduate Texts in Mathematics 5, Springer (1971).

\bibitem[P03]{P} A. Polishchuk, {\em Abelian Varieties, Theta Functions and the Fourier Transform}, Cambridge University Press, (2003).

\bibitem[RW20]{RW20} Rim\'anyi, Weber, {\em Elliptic classes of Schubert varieties via Bott-Samelson resolution}, J. Topol. 13 (2020), no. 3, 1139--1182. \MR{4100129}




\bibitem[S80]{Sie} C.L.~Siegel, {\em Advanced Analytic Number Theory}, (1980). \MR{0659851}



\bibitem[SGA3]{SGA3} 
M.~Demazure, A.~Grothendieck, 
{\em Sch\'emas en groupes I, II, III}, 
Lecture Notes in Math 151, 152, 153, Springer-Verlag, New York, 1970,
and new edition :  Documents Math\'ematiques 7, 8,  Soci\'et\'e Math\'ematique de France, 2003.  

\bibitem[S]{Stack} Stack Project authors, {\em Stack Project}, \url{
https://stacks.math.columbia.edu/}


\bibitem[ZZ22]{ZZ15}
G. Zhao and C. Zhong, {\em Representations of the elliptic affine Hecke algebras}, Advances in Mathematics
Volume 395, 24 February 2022, 108077.

\end{thebibliography}
\end{document}